\documentclass[a4paper,12pt]{article}
\usepackage{amsmath}
\usepackage{amsthm}
\usepackage{graphicx}
\usepackage{verbatim}
\usepackage{epsfig}
\usepackage{hyperref}
\usepackage{float}
\usepackage{color}
\usepackage{fullpage}
\usepackage{amsfonts}
\usepackage{amscd}
\usepackage{amssymb}
\usepackage{graphics}
\usepackage{amsmath}
\usepackage[english]{babel}
\usepackage{bbm}
\usepackage{yfonts}
\usepackage{mathrsfs}
\usepackage{color}
\usepackage{comment}
\usepackage[]{todonotes}
\usepackage[titletoc,toc,title]{appendix}
\newcounter{assumptions}

\newcommand{\BE}{{\mathbb{E}}}

\newcommand{\BP}{{\mathbb{P}}}

\newcommand{\BR}{{\mathbb{R}}}

\newcommand{\BZ}{{\mathbb{Z}}}

\newcommand{\CC}{{\mathcal{C}}}

\newcommand{\CJ}{{\mathcal{J}}}

\newcommand{\CT}{{\mathcal{T}}}

\newcommand{\mult}{{\rm mult}}

\newtheorem{theorem}{Theorem}[section]
\newtheorem{proposition}[theorem]{Proposition}
\newtheorem{lemma}[theorem]{Lemma}

\newtheorem{example}[theorem]{Example}

\begin{document}
\numberwithin{equation}{section} \numberwithin{figure}{section}
\title{Cover times of the massive random walk loop soup}
\author{Erik I. Broman{\footnote{Chalmers University of Technology and Gothenburg University, Gothenburg, Sweden, email: broman@chalmers.se, 
supported by the Swedish Research Council Grant
number: 2017-04266. }}\ and Federico Camia \footnote{New York University Abu Dhabi, Abu Dhabi, United Arab Emirates, email: federico.camia@nyu.edu}}
\maketitle
\begin{abstract}
We study cover times of subsets of $\BZ^2$ by a 
two-dimensional massive random walk loop soup. We consider a sequence 
of subsets $A_n \subset \BZ^2$ such that $|A_n| \to \infty$ and determine the distributional 
limit of their cover times $\CT(A_n).$ We allow the killing rate 
$\kappa_n$ (or equivalently the ``mass'') of the loop soup to depend on 
the size of the set $A_n$ to be covered. In particular, we determine
the limiting behavior of the cover times for inverse killing rates all the way up to
$\kappa_n^{-1}=|A_n|^{1-8/(\log \log |A_n|)},$ showing that it can be described by a Gumbel distribution. 
Since a typical loop in this model will have length at most 
of order $\kappa_n^{-1/2}=|A_n|^{1/2},$ if 
$\kappa_n^{-1}$ exceeded $|A_n|,$ the cover times of all points
in a tightly packed set $A_n$ (i.e. a square or close to a ball) would
presumably be heavily correlated, complicating the analysis. Our 
result comes close to this extreme case. \\

\end{abstract}

\section{Introduction}

In this paper we are considering a covering problem for the 
massive random
walk loop soup in $\BZ^2.$ Covering problems can be traced back to Dvoretzky
(see \cite{Dvor}) who in 1956 studied the problem of covering the circle with a
collection of randomly placed arcs of prescribed lengths. Many variants of this problem
were later studied, and we mention in particular Janson's work \cite{Janson}. Informally,
Janson fixed a set $K\subset \BR^d$ and asked for the first time when sets of small diameter
arriving according to a Poisson process cover $K$ completely. In particular, Janson determined
the asymptotic distribution of this cover time as the diameter of the covering sets shrinks to 0. 

Later, Belius \cite{Belius} took a step in a new direction when he
studied a variant of the problem in which the sets used to cover $K$ are 
unbounded. Concretely, Belius fixed a set $K\subset \BZ^d$ and considered 
so-called random interlacements arriving according to a Poisson process
with unit rate. These random interlacements can informally be understood 
as bi-infinite random walk trajectories (see \cite{Sznitman1} for 
more on this model). For this reason, the questions were posed for
$d\geq 3,$  as otherwise the random walks are recurrent.
The use of unbounded sets in the covering means that the cover times of any
two points $x,y\in K$ are dependent regardless of the distance between
$x$ and $y.$ A similar problem was then studied 
by Broman and Mussini (see \cite{BM}, which also contains references to 
other papers on coverage problems), where now $K\subset \BR^d$ and
the objects used to cover $K$ are bi-infinite cylinders. In \cite{BM}, 
the fact that $K\subset \BR^d$ is (in general) not a finite nor a 
discrete set poses a new set of challenges.

In the present paper, we restrict our attention to $\BZ^2$ and consider the so-called
{\em massive} random walk loop soup. The term {\em massive} comes from the connection
between loop soups and field theory, particularly the Gaussian free field. A random walk loop
soup with a non-zero killing rate corresponds to a Gaussian free field with a non-zero mass term.
The loops are here generated by a random walk on $\BZ^2$ which at every step has a positive
probability of being killed (or landing in a cemetery state). As long
as this killing rate is strictly positive, it keeps very large loops from appearing near the origin
and ensures a nontrivial cover time; in contrast, with a zero killing rate every vertex
in $\BZ^2$ would be instantly covered. In this sense, our current project is very much
related to the work of \cite{Belius}, as we again study the trajectory of a random walk,
but we are introducing a killing in order to keep our walks from becoming too long.
One could also study the (somewhat easier) case of finite portions of trajectories
generated by killed random walks that do not form loops, but the random walk loop soup seems more natural and interesting, in particular because of its deep connection to
the Gaussian free field and to other models of statistical mechanics
(see e.g.\ \cite{Camia} and references therein). Aside from this connection, the random
walk loop soup is also an object of intrinsic interest as a prototypical example of a Poissonian
system amenable to rigorous analysis thanks to the vast body of knowledge on the behavior
of the random walk in two dimensions. We remark that while the 
usual set-up for the random walk loop soup is for the case
of finite graphs (see for instance 
\cite{LeJan} and \cite{Sznitman}), the only thing which is really 
needed is for the Green's function to be finite. In our full-space
setting, i.e. on $\BZ^2$, this is accomplished by having a non-zero
killing rate $\kappa$, see further the remark in Section \ref{sec:model}.

We will give a precise definition of the massive random walk loop soup in Section \ref{sec:model},
but in order to present our main results we give here a short informal explanation.
We will consider a measure $\mu$ on the set of all loops
(i.e. finite walks on $\BZ^2$ ending at the same vertex where they started). Because of the
non-zero killing rate, the measure $\mu$ does not give much weight to very long loops.
In particular, $\mu(\Gamma_x)<\infty$ where $\Gamma_x$ denotes the set of
loops containing $x\in \BZ^2.$ Furthermore, this measure is translation 
invariant so that, in particular, $\mu(\Gamma_x)=\mu(\Gamma_o)$ for every
$x \in \BZ^2,$ where $o$ denotes the origin of the square lattice. Since the quantity 
$\mu(\Gamma_o)$ will play a central role, we point out already here that it is a 
function of 
the Green's function at 0 of a random walk with killing rate $\kappa.$ 
Furthermore, as our main result (Theorem \ref{thm:main} below) concerns the 
case where $\kappa$ goes to 0
with the sizes of the sets we are covering, it follows that 
$\mu(\Gamma_o)$ implicitly 
depends on the sizes of the sets we are covering. 

The model that we study
here is then a Poisson process $\omega$ on $\Gamma \times [0,\infty),$ where 
$\Gamma=\bigcup_{x\in \BZ^2} \Gamma_x$ is the set of all loops. Furthermore, a pair
$(\gamma,s)\in \omega$ is a loop $\gamma$ along with a ``time-stamp'' $s$
denoting the time at which loop $\gamma$ arrives. We can then define the cover time of the set $A\subset \BZ^2$ by letting
\[
\CT(A):=\inf\left\{t>0: 
A \subset \bigcup_{(\gamma,s)\in \omega:s \leq t} \gamma\right\},
\]
where we abuse notation somewhat and identify the loop $\gamma$ with 
its trace, i.e. the vertices $x\in \BZ^2$ that it encounters. Our main result
concerns the asymptotic cover time of a growing sequence of sets 
$(A_n)_{n \geq 1},$ as follows.
\begin{theorem} \label{thm:main}
Consider a sequence of finite subsets $A_n \subset \BZ^2$ such that 
$|A_n| \uparrow \infty.$ Furthermore, assume that the killing rates $\kappa_n$ are 
such that, for every $n,$
$\exp(e^{32})\leq \kappa^{-1}_n\leq |A_n|^{1-8/(\log \log |A_n|)}.$
We then have that for $n$ large enough
\begin{equation}\label{eqn:basicineq}
\sup_{z \in \BR}|\BP(\mu(\Gamma_o)\CT(A_n)-\log |A_n|\leq z)
-\exp(-e^{-z})|
\leq 12|A_n|^{-\frac{1}{800 \mu(\Gamma_o)}}
\end{equation}
and therefore
\begin{equation}\label{eqn:mainthm2}
\mu(\Gamma_o)\CT(A_n)-\log |A_n| \underset{n \to \infty}{\longrightarrow} G \; \text{ in distribution,}
\end{equation}
where $G$ is a Gumbel distributed random variable. 
\end{theorem}
\noindent
{\bf Remarks:}
It may seem that the bound on the right hand side of \eqref{eqn:basicineq}
does not depend on the killing rates $\kappa_n.$ 
However, equations \eqref{eqn:muGammaoest} and 
\eqref{eqn:muGammaolowest} show that, for $\kappa_n^{-1}$ large,
$\mu(\Gamma_o)$ is such that
$\vert \mu(\Gamma_0)-\log\log \kappa_n^{-1} \vert < 2$ (where the constant
$2$ is rather arbitrary). Furthermore, the remark 
after the proof of Lemma \ref{lemma:doublesumsmalldist} indicates that 
it may not be possible to drastically improve on the rate of 
convergence in \eqref{eqn:basicineq}, at least not by using the methods of this paper.

We assume the lower bound on $\kappa_n^{-1}$ in the statement of 
Theorem \ref{thm:main} out of convenience, and we claim that this can 
be relaxed (at least somewhat) by adding further details to our proofs. 
However, as we deem it natural to let $\kappa_n\to 0$ as 
$|A_n|\to \infty,$ we do not think it worthwhile to make the paper more
technical than it is in order to improve on this lower bound.

Furthermore, the upper bound on $\kappa_n^{-1}$ cannot be easily
improved, at least not substantially. Indeed, the discussion after 
the proof of 
Lemma \ref{lemma:doublesummedest2} indicates that while it may be 
possible to improve the upper bound by replacing the number 8 
with a slightly lower number, improving it further would require 
new ideas, if at all possible (see further the discussion below).

\medskip

\noindent
We continue this section with a more in depth discussion of our 
main result and of its proof.

It is straightforward to determine (see the start of 
Section \ref{sec:probabilities}) that the expected number of 
uncovered vertices $x\in A_n$ at time $\frac{\log |A_n|}{\mu(\Gamma_o)}$
is exactly $1$, and this is why the distributional limit  
in Theorem \ref{thm:main} may exist at all. Furthermore, it is easy 
to show (see \eqref{eqn:onepointdistribution}) that 
$\mu(\Gamma_o)\CT(o)$ is an exponentially distributed random variable 
with parameter 1. By using the well known fact that the maximum 
of independent exponentially distributed random variables (with 
parameters all equal to $1$) converges to a Gumbel distribution, 
we see that \eqref{eqn:basicineq} very much corresponds to the 
situation where the vertices of $A_n$ are covered independently.
Indeed, \eqref{eqn:mainthm2} means that
$\CT(A_n) \approx \frac{\log |A_n|+G}{\mu(\Gamma_o)}$ and, since 
$\lim_{\kappa \to 0}\mu(\Gamma_o)=\infty,$ we see that the cover time
will be concentrated around $\frac{\log |A_n|}{\mu(\Gamma_o)}$
with extremely small fluctuations of size $\frac{G}{\mu(\Gamma_o)}$.

Our main result is not surprising for large killing rates
(such as when $\kappa_n$ is constant). This is because
in such a regime large loops are strongly suppressed, which makes our model look similar to those studied
in \cite{BM}, \cite{Belius} and \cite{Janson}, where similar results are obtained.
The main difference in our work is that we let the killing rate 
$\kappa_n$ go to 0 as the size of the set that needs to be covered diverges.

In this situation, we expect that the behavior may depend on the geometry
(to be more precise, on the sparsity) of the sets $A_n,$ as well as on
how fast the killing rates $\kappa_n$ go to 0. 
To see this, assume first that we are in the ``compact'' case where $A_n$ is
(as close as possible to) a ball of radius $\sqrt{n}$,
and consider the following situations.  
\begin{itemize}
	\item[(i)] If $\kappa_n$ approaches zero very quickly,
	then for $n$ large, the diameter of 
	a typical loop intersecting $A_n$ is vastly larger than the linear
	size of $A_n$. It is then
	natural to expect that the first loop that arrives and touches 
	$A_n$ will in fact
	cover $A_n$ completely, and the re-scaled cover time will simply 
	be an exponential random variable as $n \to \infty.$ We give 
	further support to this claim in Section \ref{sec:examples} which
	includes a longer discussion along with two simple examples 
	(Examples \ref{ex:twopointwidesep} and \ref{ex:twopointneigh})
	further mentioned below.
	
	\item[(ii)] If instead $\kappa_n=|A_n|^{-\alpha}$ for some 
	$\alpha<1$, then for $n$ large, the diameter of 
	a typical loop intersecting $A_n$ is much smaller than the 
	linear size of $A_n.$ This will create enough independence 
	for the re-scaled cover time to converge to a Gumbel distribution
	as indeed Theorem \ref{thm:main} shows.

	\item[(ii)] If instead $\kappa_n \to 0$ at a rate which
 	is in between the two cases above, then the diameter of 
	a typical loop may still be larger than $A_n,$ but it will most 
	likely not cover the entirety of $A_n.$ It is not clear to us 
	how the cover time will behave in this intermediate case.
	
\end{itemize}

If $A_n$ is sparser or stretched (say in the form of a 
line of length $n$), then the potential 
different phases described above may simply occur at 
other thresholds. However, if we allow the 
separation between points to depend on the killing rate,
we can easily create an example (see Example 
\ref{ex:manypointswidesep})
in which the limit of the cover time is always a Gumbel. 

Having argued that one expects a different type of behavior for high and low killing rates, it is natural to ask where the threshold between those two regimes lies. For the ``compact'' case illustrated above,
one may guess that the threshold could be when 
$\kappa_n \sim |A_n|^{-1}$, so that the linear size of $A_n$ in 
this case is of the same order as the diameter of a typical large loop, which essentially corresponds to the \emph{correlation length} of the system (the separation distance at which two parts of the system become roughly independent). We remark that our main result comes very close to 
this supposed threshold. It may of course also be the case that the
correct threshold corresponds to an even quicker rate at which 
$\kappa_n \to 0,$ but if so, other methods than the ones employed in
this paper will be needed to get close to the threshold.

We note that, if one takes a scaling limit, re-scaling space by $1/\sqrt{A_n}$ and time by $1/{A_n}$, $\kappa_n=|A_n|^{-1}$ corresponds to the \emph{near-critical regime} and leads to a \emph{massive} Brownian loop soup (see \cite{Camia}).
In contrast to this, if $\kappa_n=|A_n|^{-\alpha}$ for some
$\alpha<1$ one expects the scaling limit to be trivial (meaning that no macroscopic loops survive), while if $\alpha>1$ one expects to obtain the \emph{critical} (i.e.\ scale invariant) Brownian loop soup (see \cite{Camia} for further discussion).

We briefly mention Examples \ref{ex:twopointwidesep} 
and \ref{ex:twopointneigh}, both considering sets $A_n$ consisting
of only two points. In Example \ref{ex:twopointwidesep}
the two points are
vastly separated, and the re-scaled cover time is shown to be 
the maximum of two independent exponential random variables. In 
contrast, Example \ref{ex:twopointneigh} deals with two points
that are neighbors, and the re-scaled cover time is shown to be 
a single exponential random variable.
This provides some additional support for the discussion above concerning the potential
different phases.

\medskip

The overall strategy of the proof of Theorem \ref{thm:main} can be informally
described as follows. At a time just before the expected cover time,
i.e. at time $(1-\epsilon)\frac{\log |A_n|}{\mu(\Gamma_o)}$, the not yet
covered region should consist of relatively few and well separated
vertices (see Proposition \ref{prop:HAeps}). These separated vertices 
will then be covered ``almost independently'' as the distances between
them are so large that we will  not see many loops that are large enough
to hit two such vertices. 
The main work will go into establishing the first of these two steps,
for which we will need to perform an involved and detailed second
moment estimate. Although the general strategy that we will follow 
has been used in \cite{Belius} and \cite{BM}, the main part of the
work and challenges here are different. This is intimately 
connected to the fact that the methods must be fine-tuned in 
order for Theorem \ref{thm:main} to work as close as possible to the
case $\kappa_n^{-1}=|A_n|.$ 

We end this introduction with an outline of the rest of the paper. 
In Section \ref{sec:model} we will define and discuss the random
walk loop soup. In Section \ref{sec:Green} we will obtain various 
estimates on the Green's function. However, to avoid breaking the flow of the paper, many of the calculations used
in this section will be deferred to an appendix (Appendix 
\ref{app:walks}). The results of Section \ref{sec:Green} will 
then be used in Section \ref{sec:probabilities} in order to obtain 
estimates on the probabilities involved in our second moment estimate.
The latter is done in Section \ref{sec:2ndmom} and in turn, 
these results are 
used in Section \ref{sec:proofofmain} to prove our main result.
Finally, Section \ref{sec:examples} contains the three examples 
and the discussion mentioned above.

\section{The loop soup} \label{sec:model}
The purpose of this section is twofold. Firstly, it will serve to introduce 
necessary notation and definitions. Secondly, it will serve as a brief 
introduction to random walk loop soups in the particular case that we
study in this paper. See also \cite{LT,LeJan,Sznitman,Camia} for 
an overview of the model studied in this paper.

We consider a discrete time simple symmetric random walk loop 
soup in $\BZ^2$ with killing rate $\kappa>0.$ In this setting, a 
walker positioned at $x$ at time $n$ will move to a neighbor $y$ 
with probability $1/(4+\kappa)$, and it will be killed (or equivalently 
moved to a cemetery state) with probability $\kappa/(4+\kappa).$ 
Next, $\gamma_r$ is a loop rooted at the vertex $x_0$ and of length 
$|\gamma_r|=N$ if 
\[
\gamma_r=((\gamma_r)_0,\ldots,(\gamma_r)_{N-1})
=(x_0,x_1,\ldots,x_{N-1}),
\]
for any sequence of neighboring vertices $x_0,\ldots,x_{N-1}$ 
where $x_{N-1}$ is a vertex neighboring $x_0.$ As usual 
(see \cite{LeJan} or \cite{Sznitman}), we only consider non-trivial loops,
i.e.\ only loops with $N\geq 2.$

We note that while the 
alternative notation $\gamma_r=(x_0,x_1,\ldots,x_{N-1},x_0)$ may seem
more natural (as it ``closes the loop''), it would be more cumbersome 
when we want to consider time-shifts of the loops. One could of course 
also define the loop in terms of the edges traversed, but as we consider 
the cover times for vertices, this seems less natural. 

Since we are considering loops in $\BZ^2,$ we must have that
$|\gamma_r|$ is an even number. 
The rooted measure $\mu_r$ of a fixed rooted loop $\gamma_r$ is then defined
to be 
\[
\mu_r(\gamma_r:|\gamma_r|=2n,(\gamma_r)_0
=x_0,(\gamma_r)_1=x_1,\ldots,(\gamma_r)_{2n-1}
=x_{2n-1})
=\frac{1}{2n}\left(\frac{1}{4+\kappa}\right)^{2n},
\]
for $n\geq 1.$ We see that $\mu_r(\gamma_r)$ is the probability
of the corresponding killed random walk on $\BZ^2$ multiplied by the 
factor $1/(2n)$. Intuitively, the reason for this modification is 
that (most) loops have $2n$ possible starting points and will therefore
contribute $2n$ times in the Poissonian construction below.

Proceeding, we find that the total rooted measure of all loops rooted 
at $x_0$ of length $2n$ becomes
\[
\mu_r(\{\gamma_r:|\gamma_r|=2n,(\gamma_r)_0=x_0\})
=\frac{L_{2n}}{2n}\left(\frac{1}{4+\kappa}\right)^{2n},
\]
where $L_{2n}$ denotes the number of loops rooted at $o$ and of length $2n.$ 

In order to define the (unrooted) loop measure we start by defining 
equivalence classes of loops by saying that the rooted 
loops $\gamma_r, \gamma_r'$ are equivalent 
if we can obtain one from the other by a time-shift. More formally, 
if $|\gamma_r|=2n$, then $\gamma_r \sim \gamma_r'$ if there 
exists some $0\leq  m <2n$ such that 
\[
((\gamma_r)_0,\ldots,(\gamma_r)_{2n-1})
=((\gamma'_r)_m, \ldots, (\gamma'_r)_{2n-1},(\gamma'_r)_{0},\ldots,(\gamma'_r)_{m-1}).
\]
We see that the equivalence class of $\gamma_r$ contains exactly 
$\frac{2n}{\mult(\gamma_r)}$ rooted loops. Here, $\mult(\gamma_r)$ is
the largest $k$ such that $\gamma_r$ can be written as the concatenation of $k$ identical loops.
We will think of an equivalence class as an unrooted loop (i.e. \ as a sequence of
vertices with a specified order but with no specified first vertex), and we shall 
denote such a loop by $\gamma.$ We will occasionally write 
$\gamma_r \in \gamma$ to indicate that the rooted loop $\gamma_r$ is a member
of the equivalence class $\gamma$.

We then define the (unrooted) measure $\mu$ on loops by letting
$\mu(\gamma)$ equal the weight of the rooted measure for a member of the 
equivalence class of $\gamma,$ multiplied by the number of members in this 
equivalence class. That is, 
\begin{eqnarray} \label{eqn:mudef}
\lefteqn{\mu(\gamma)=\sum_{\gamma_r\in \gamma}\mu_r(\gamma_r)
=\sum_{\gamma_r\in \gamma} \frac{1}{2n}\left(\frac{1}{4+\kappa}\right)^{2n}} \\
& & =\frac{1}{2n}\left(\frac{1}{4+\kappa}\right)^{2n}\frac{2n}{\mult(\gamma)}
=\left(\frac{1}{4+\kappa}\right)^{2n}\frac{1}{\mult(\gamma)}, \nonumber
\end{eqnarray}
where $\mult(\gamma)=\mult(\gamma_r)$ for any (and therefore every) $\gamma_r \in \gamma.$
Equation \eqref{eqn:mudef} thus defines our measure $\mu.$
We choose to work with unrooted loops and the corresponding measure because this is the canonical choice (see \cite{LeJan,Sznitman}) leading to the Brownian loop soup in the scaling limit \cite{LT}.

We now let $\Gamma_{x}^{2n}$ denote the set of all unrooted loops $\gamma$
such that $x\in \gamma$ and $|\gamma|=2n.$ Then, we define 
\[
\Gamma_{x}=\bigcup_{n=1}^\infty \Gamma_{x}^{2n}.
\]
We observe that
\begin{equation}\label{eqn:muGammao}
\mu(\Gamma_o)
=\sum_{n=1}^\infty \sum_{\gamma \in \Gamma_o^{2n}} \mu(\gamma)
=\sum_{n=1}^\infty \sum_{\gamma \in \Gamma_o^{2n}}
\left(\frac{1}{4+\kappa}\right)^{2n}\frac{1}{\mult(\gamma)}.
\end{equation}
It will turn out that the quantity $\mu(\Gamma_o)$ will play an essential role
in the rest of the paper. However, while \eqref{eqn:muGammao} gives a concrete
and easily understandable expression for $\mu(\Gamma_o)$ it is not the most 
useful, and we will instead use \eqref{eqn:mulogG} below.

Returning to our measure $\mu$ we now let $\omega$ denote a Poisson
point process on $\Gamma \times [0,\infty)$ with intensity measure 
$\mu \times dt$. Here, $\Gamma=\bigcup_{x\in \BZ^2} \Gamma_x$ simply 
denotes the set of all unrooted loops in 
$\BZ^2.$ We shall think of a pair $(\gamma,t)\in \omega$  as a loop
$\gamma$ along with a ``time-stamp'' $t$ which corresponds to the time
at which the loop arrived. We also let 
\[
\omega_t:=\{\gamma\in \Gamma: (\gamma,s)\in \omega \textrm{ for some } s\leq t\},
\]
so that $\omega_t$ is the collection of loops that have arrived before time 
$t.$ It will be convenient to introduce the notation 
\[
\CC_t=\{x\in \BZ^2:x\in \gamma \textrm{ for some } \gamma\in \omega_t\},
\]
so that $\CC_t$ is the covered region at time $t>0.$

\medskip
\noindent
{\bf Remark:}
In this section, and in the rest of the paper, we will use equations
such as \eqref{eqn:probGreenbasic} below from \cite{LeJan} and 
\cite{Sznitman} involving the Green's function.
This requires a comment, since in \cite{LeJan} and 
\cite{Sznitman}, those equations are derived working with 
finite graphs, while we consider the infinite square lattice.
The use of finite graphs in \cite{LeJan} and 
\cite{Sznitman} is largely a matter of convenience 
(\cite{LeJanpersonal}), since it allows us to write 
explicit formulas in terms of determinants of finite matrices. 
Nevertheless, the final formulas in terms of the Green's function are 
valid whenever the Green’s function is well defined.

To see this why this is the case in the specific example of the 
massive random walk loop soup on the square lattice, the reader 
can think of coupling a massive loop soup on $\mathbb{Z}^2$ and a 
loop soup on $[-L,L]^2 \cap \mathbb{Z}^2$, obtained from the first 
one by removing all loops that exit $[-L,L]^2$. If one focuses on the 
restrictions of the two processes to a finite window 
$[-L_0,L_0]^2 \cap \mathbb{Z}^2$, for any fixed $L_0$, and sends 
$L \to \infty$, the presence of a positive killing rate implies 
that the restriction of the second process to 
$[-L_0,L_0]^2 \cap \mathbb{Z}^2$ converges to that of the first.
On the other hand, for the second process, one can use the 
formulas from \cite{LeJan} and \cite{Sznitman} for any finite $L$. 
Moreover, the expressions in those formulas converge as 
$L \to \infty$ because the Green’s function stays finite in that 
limit, due to the positive killing rate.

\medskip

As just remarked, there is a close connection between the loop soup 
and the Green's function
for the killed simple symmetric random walk on $\BZ^2$.  
It is known (see (4.18) on p. 74 of \cite{Sznitman} or p. 45 of \cite{LeJan}) that 
(for the simple symmetric random walk with killing rate $\kappa$)
\begin{equation}\label{eqn:probGreenbasic}
\BP(x\cap \CC_u=\emptyset)
=\BP(\not \exists \gamma \in \omega_u: o\in \gamma)
=\exp(-u\mu(\Gamma_o)) 
=\left(\frac{1}{(4+\kappa)g(o,o)}\right)^u
\end{equation}
where the first equality follows from the construction of the Poisson 
process and translation invariance.

Here, $g(x,y)$ is a Green's function given by (see \cite{Sznitman} 
p. 9, (1.26))
\[
g(x,y)=\int_0^\infty \frac{1}{4+\kappa} \BP(X_t^x=y) dt
\]
where $(X^x_t)_{t>0}$ is a continuous time random walk (started at $x$)
which waits an 
exponential time with parameter 1 and then picks any neighbor with 
probability $1/(4+\kappa)$ and is killed with probability $\kappa/(4+\kappa).$
Clearly, if $N_t$ denotes the number of steps that this random walk has 
taken by time $t,$ we then have that 
\begin{eqnarray*} 
\lefteqn{g(x,y)=\int_0^\infty \frac{1}{4+\kappa} \BP(X_t^x=y) dt}\\
& & =\frac{1}{4+\kappa}\int_0^\infty  
\sum_{n=0}^\infty \BP(X_t^x=y |N_t=n)\BP(N_t=n)dt \nonumber\\
& & =\frac{1}{4+\kappa} \sum_{n=0}^\infty
\int_0^\infty  \BP(S_n^{x,\kappa}=y)\frac{t^n}{n!}e^{-t}dt \nonumber \\
& & =\frac{1}{4+\kappa} \sum_{n=0}^\infty \BP(S_n^{x,\kappa}=y)
\int_0^\infty  \frac{t^n}{n!}e^{-t}dt
=\frac{1}{4+\kappa} \sum_{n=0}^\infty \BP(S_n^{x,\kappa}=y), \nonumber
\end{eqnarray*}
where $S_n^{x,\kappa}$ denotes a discrete time random walk started at $x$
and with killing rate $\kappa.$

Combining the above we obtain the formula
\begin{equation} \label{eqn:probGreen}
\BP(x\cap \CC_u=\emptyset)
=(G^{o,o})^{-u},
\end{equation}
where 
\[
G^{x,y}=\sum_{n=0}^\infty \BP(S_n^{x,\kappa}=y).
\]
We note also that from \eqref{eqn:probGreenbasic} and
\eqref{eqn:probGreen} we have that
$e^{-\mu(\Gamma_o)}=\frac{1}{G^{o,o}}$ and so  
\begin{equation} \label{eqn:mulogG}
\mu(\Gamma_o)=\log G^{o,o}.
\end{equation}
As mentioned above, this equation will turn out to be much more useful 
for us than \eqref{eqn:muGammao}. Observe that 
$\omega_u\cap \Gamma_x \cap \Gamma_y$ is the set of loops in 
$\omega_u$ which intersect both $x$ and $y.$ We have 
(according to \cite{LeJan}, p. 45) that
\begin{eqnarray}\label{eqn:LeJanident}
\lefteqn{\BP(\omega_u\cap \Gamma_x \cap \Gamma_y=\emptyset)}\\
& & =\exp(-u \mu(\Gamma_x \cap \Gamma_y))
=\left(1-\left(\frac{g(x,y)}{g(o,o)}\right)^2\right)^u
=\left(1-\left(\frac{G^{x,y}}{G^{o,o}}\right)^2\right)^u. \nonumber
\end{eqnarray}
We conclude that 
\begin{eqnarray} \label{eqn:probxyuncov}
\lefteqn{\BP(\{x,y\} \cap \CC_u=\emptyset)
=\exp(-u \mu(\Gamma_x \cup \Gamma_y))
=\exp(-2u\mu(\Gamma_x)+u\mu(\Gamma_x \cap \Gamma_y))} \nonumber \\
& & 
=\frac{\BP(x\cap \CC_u=\emptyset)^2}{\BP(\not \exists \gamma \in \omega_u: x,y\in \gamma)}
=\frac{\left(\frac{1}{G^{o,o}}\right)^{2u}}
{\left(1-\left(\frac{G^{x,y}}{G^{o,o}}\right)^2\right)^u}
=
\left(\frac{1}{(G^{o,o})^2-(G^{x,y})^2}\right)^u.
\end{eqnarray}
Much of the effort of this paper will be focused around obtaining good
estimates for \eqref{eqn:probxyuncov}, and for other similar 
quantities. For this reason we shall need to study some aspects of the 
Green's function $G^{x,y}$ in detail, and then use these results to obtain 
good estimates of probabilities such as 
$\BP(\{x,y\} \cap \CC_u=\emptyset).$
In order to structure this, we choose to devote Section \ref{sec:Green}
exclusively to results concerning Green's functions. These results 
are then used in Section \ref{sec:probabilities} to obtain our estimates
for relevant probabilities.

\section{Green's function estimates} \label{sec:Green}
We will write $S_n$ in place of $S_n^{o,o}$ and we start by observing that 
\begin{equation}\label{eqn:Goxexpress}
G^{o,x}=\sum_{n=0}^\infty \BP(S^{o,\kappa}_{n}=x)
=\sum_{n=|x|}^\infty \left(\frac{4}{4+\kappa}\right)^{n}\BP(S_{n}=x)
=\sum_{n=|x|}^\infty\left(\frac{1}{4+\kappa}\right)^{n}W^{o,x}_{n}
\end{equation}
where $W^{o,x}_n$ denotes the number of walks of length $n$ starting
at the origin and ending at~$x.$
In \eqref{eqn:Goxexpress} and in the rest of the paper, for
$x=(x_1,x_2) \in \mathbb{Z}^2$, $|x|$ denotes $|x_1|+|x_2|$.
It is clear from \eqref{eqn:probxyuncov} that in order to 
bound $\BP(\{x,y\} \cap \CC_u=\emptyset)$ we should strive to 
find good estimates for $G^{o,o},G^{o,x}$ and the difference
\begin{equation} \label{eqn:Gdiff}
G^{o,o}-G^{o,x}
=\sum_{n=0}^\infty\left(\frac{1}{4+\kappa}\right)^{n}W^{o,o}_{n}
-\sum_{n=|x|}^\infty\left(\frac{1}{4+\kappa}\right)^{n}W^{o,x}_{n},
\end{equation}
and this is the main purpose of this section.
The main results of the current section are Propositions 
\ref{prop:Greendiffest1} and \ref{prop:Goxest1}.
Proposition \ref{prop:Greendiffest1} will 
give estimates on \eqref{eqn:Gdiff} for small and moderate values of 
$|x|$, while Propositions \ref{prop:Goxest1} 
will provide estimates on $G^{o,x}$ for large values of $|x|.$ 
We mention that Propositions \ref{prop:Goxest1}
is somewhat specialized to work for large values of $\kappa^{-1}.$ 

To avoid breaking the flow of the paper, the proofs of elementary 
lemmas concerning the number of walks $W_n^{o,x},$ and estimates 
on partial sums of \eqref{eqn:Goxexpress}, 
are deferred to Appendix \ref{app:walks}.

We will now focus on 
Proposition \ref{prop:Greendiffest1}, which will be proved  
through two lemmas. Firstly, Lemma \ref{lemma:WxWo} will allow us
to estimate $G^{o,o}-G^{o,x}$ in terms of partial sums of $G^{o,o}.$
Then, we will use Lemma \ref{lemma:partialsumlowest} to quantify 
these bounds. In order to prove Proposition \ref{prop:Goxest1}
we will use a consequence (Lemma \ref{lemma:sumWoxest}) of the 
local central limit theorem, along with Lemmas 
\ref{lemma:WxWo} and \ref{lemma:partialsumupest}.

We can now state our first lemma which is proved in Appendix 
\ref{app:walks}.
\begin{lemma} \label{lemma:WxWo}
For any $x\in \BZ^2$ such that $|x|$ is even, we have that 
\[
W_{2n}^{o,x}\leq W_{2n}^{o,o}
\]
for every $n\geq 0.$
\end{lemma}
Our second lemma (again proved in Appendix \ref{app:walks}) 
provides lower bounds on the partial sums of 
\[
G^{o,o}=\sum_{n=0}^{\infty} 
\left(\frac{1}{4+\kappa}\right)^{2n}W_{2n}^{o,o}.
\]

\begin{lemma} \label{lemma:partialsumlowest}
For any $\kappa>0$ and any $N\geq 1$ we have that 
\begin{equation} \label{eqn:partialsumlowest1}
\sum_{n=0}^{N-1} \left(\frac{1}{4+\kappa}\right)^{2n}W^{o,o}_{2n}
\geq 1+\frac{\log N}{\pi}-\frac{N\kappa}{\pi}-\frac{1}{3 \pi}.
\end{equation}
Furthermore, 
\begin{equation} \label{eqn:Goolowest}
G^{o,o}\geq \frac{\log \kappa^{-1}}{\pi}+1-\frac{4}{3 \pi}.
\end{equation}
\end{lemma}

We are now ready to state and prove our first result concerning 
the difference \eqref{eqn:Gdiff}. Proposition 
\ref{prop:Greendiffest1} will be ``basic'' in the sense that 
the statements are not the strongest possible, but they are sufficient 
for our purposes.
Later, we will prove Proposition \ref{prop:Goxest1} which will 
be more specialized.
\begin{proposition} \label{prop:Greendiffest1}
For any $\kappa>0$ we have that 
\begin{equation} \label{eqn:GoxGooratiolim}
\lim_{|x|\to \infty} G^{o,x} = 0.
\end{equation}
For any $|x|\geq 1$ we have that  
\begin{equation} \label{eqn:Greengenlest}
G^{o,o}-G^{o,x} \geq \frac{3}{4}.
\end{equation}
If $4\leq |x|\leq 2\kappa^{-1},$ and $\kappa^{-1}\geq 2,$ 
we have that 
\begin{equation} \label{eqn:Greenmedest}
G^{o,o}-G^{o,x} \geq \frac{\log |x|}{\pi}.
\end{equation}

\end{proposition}
\noindent
\begin{proof}
The first result, i.e.\ \eqref{eqn:GoxGooratiolim}, is an immediate consequence
of (3.1) since we clearly have that 
\[
G^{o,x}
=\sum_{n=|x|}^\infty \left(\frac{4}{4+\kappa}\right)^{n}\BP(S_{n}=x)
\leq \sum_{n=|x|}^\infty\left(\frac{4}{4+\kappa}\right)^{n} \to 0,
\] 
as $|x|\to \infty.$

We now turn to \eqref{eqn:Greenmedest} and we will 
also assume (momentarily) that $|x|$ is an even number.
We have that 
\begin{equation}\label{eqn:GooGoxdiff}
G^{o,o}-G^{o,x}
= \sum_{n=0}^\infty \left(\frac{1}{4+\kappa}\right)^{2n} 
(W_{2n}^{o,o}-W_{2n}^{o,x})
\geq \sum_{n=0}^{\frac{|x|}{2}-1} \left(\frac{1}{4+\kappa}\right)^{2n} 
W_{2n}^{o,o},
\end{equation}
by using Lemma \ref{lemma:WxWo} and the fact that $W_{2n}^{o,x}=0$
for $n\leq |x|/2-1.$ We can then use \eqref{eqn:partialsumlowest1} 
with $N=|x|/2$ to see that 
\begin{eqnarray*}
\lefteqn{G^{o,o}-G^{o,x}
\geq 1+\frac{\log(|x|/2)}{\pi}-\frac{\kappa |x|/2}{\pi}
-\frac{1}{3 \pi}}\\
& & \geq \frac{\log |x|}{\pi}+1-\frac{\log 2}{\pi}
-\frac{1}{\pi}-\frac{1}{3\pi}
\geq \frac{\log |x|}{\pi},
\end{eqnarray*}
where we used the assumption that $|x|\leq 2 \kappa^{-1}$ in the 
second inequality. This proves \eqref{eqn:Greenmedest} in the case 
when $|x|$ is even. 

We will have to take some extra care when $|x|$ is odd.
Therefore, assume that $x=(2l+1,2k)$ with 
$5\leq |x|\leq 2\kappa^{-1}$ and 
observe that by \eqref{eqn:Goxexpress},
\begin{eqnarray*}
\lefteqn{G^{o,(2l+1,2k)}
=\sum_{n=0}^\infty\left(\frac{1}{4+\kappa}\right)^{n}W^{o,(2l+1,2k)}_{n}}
 \\
& & =\sum_{n=0}^\infty\left(\frac{1}{4+\kappa}\right)^{n}
\left(W^{o,(2l,2k)}_{n-1}+W^{o,(2l+2,2k)}_{n-1}
+W^{o,(2l+1,2k-1)}_{n-1}+W^{o,(2l+1,2k+1)}_{n-1}\right)\\
& & =\frac{1}{4+\kappa} \left(G^{o,(2l,2k)}+G^{o,(2l+2,2k)}
+G^{o,(2l+1,2k-1)}+G^{o,(2l+1,2k+1)}\right).
\end{eqnarray*}
We then conclude that 
\begin{eqnarray*}
\lefteqn{G^{o,o}-G^{o,(2l+1,2k)}}\\
& & \geq \frac{1}{4}\left(4G^{o,o}-G^{o,(2l,2k)}-G^{o,(2l+2,2k)}
-G^{o,(2l+1,2k-1)}-G^{o,(2l+1,2k+1)}\right) \\
& & \geq \frac{1}{2} \sum_{n=0}^{(|x|-1)/2-1} 
\left(\frac{1}{4+\kappa}\right)^{2n}W_{2n}^{o,o}
+\frac{1}{2}\sum_{n=0}^{(|x|+1)/2-1} 
\left(\frac{1}{4+\kappa}\right)^{2n}W_{2n}^{o,o},
\end{eqnarray*}
since at most two of the neighbors of $x$ are closer to $o$ than $x$. 
By again using \eqref{eqn:partialsumlowest1} we now see that 
\begin{eqnarray*}
\lefteqn{G^{o,o}-G^{o,(2l+1,2k)}}\\
& & \geq \frac{1}{2} 
\left(1+\frac{\log (|x|-1)-\log 2}{\pi}
-\frac{(|x|-1)\kappa}{2 \pi}-\frac{1}{3 \pi}\right)\\
& & \hspace{4mm}+\frac{1}{2}\left(1+\frac{\log (|x|+1)-\log 2}{\pi}
-\frac{(|x|+1)\kappa}{2 \pi}-\frac{1}{3 \pi}\right) \\
& & =\frac{\log (|x|^2-1)}{2\pi}+1-\frac{\log 2}{\pi}
-\frac{|x| \kappa}{2 \pi}-\frac{1}{3 \pi}.
\end{eqnarray*}
Furthermore, we have that $y^2-1\geq \frac{8y^2}{9}$ for every $y\geq 3$
and therefore, 
\begin{eqnarray} \label{eqn:averaging}
\lefteqn{G^{o,o}-G^{o,(2l+1,2k)}}\\
& & \geq \frac{\log \frac{8}{9}+\log |x|^2}{2\pi}
+1-\frac{\log 2}{\pi}
-\frac{|x| \kappa}{2 \pi}-\frac{1}{3 \pi}
\nonumber \\
& & \geq \frac{\log |x|}{\pi} +1+\frac{\log 8-\log 9-\log 4}{2 \pi}
-\frac{1}{\pi}-\frac{1}{3\pi}
\geq \frac{\log |x|}{\pi}, \nonumber
\end{eqnarray}
where we used that $|x|\leq 2\kappa^{-1}$ in the penultimate inequality.
By symmetry, the same estimate holds when $G^{o,(2l+1,2k)}$ is 
replaced by $G^{o,(2l,2k+1)},$
and this establishes \eqref{eqn:Greenmedest}.

For \eqref{eqn:Greengenlest}, consider first the case when $x=(1,0)$ 
and observe that, as above,
\[
G^{o,o}-G^{o,(1,0)} \geq \frac{1}{4}\left(4G^{o,o}-G^{o,o}-G^{o,(2,0)}
-G^{o,(1,1)}-G^{o,(1,-1)}\right)
\geq \frac{3}{4},
\]
where we used \eqref{eqn:GooGoxdiff} to conclude that 
$G^{o,o}-G^{o,x}\geq W_0^{o,o}=1$ whenever $|x|\geq 2$ is even.
The statement then follows for all $|x|=1$ by symmetry.
Next, if $x$ such that $|x|\geq 2$ is even, we again observe that 
by \eqref{eqn:GooGoxdiff} $G^{o,o}-G^{o,x}\geq W_0^{o,o}=1.$
For odd values of $|x|\geq 3$ we can sum over the neighbors 
to reach the same conclusion.
\end{proof}

Our next lemma will give upper bounds on the tails of the sums in 
$G^{o,o}.$ The first part (i.e. \eqref{eqn:Wootailest}) will be 
used to prove Proposition \ref{prop:Goxest1}, while 
the second part (i.e. \eqref{eqn:Wootailest_alt}) will be used to prove
Propositions \ref{prop:Goxest2} and \ref{prop:Goxest1},
and the last part (i.e. \eqref{eqn:Gooupest}) will be used in later
sections. The proof is again deferred until Appendix \ref{app:walks}.

\begin{lemma} \label{lemma:partialsumupest}
For any $0<\kappa<1,$ and any $N\in\{1,2,\ldots\}$ such that 
$N\kappa<1$, we have that 
\begin{equation} \label{eqn:Wootailest}
\sum_{n=N+1}^\infty \left(\frac{1}{4+\kappa}\right)^{2n} W_{2n}^{o,o}
\leq \frac{\log (N\kappa)^{-1}}{\pi}+\frac{1}{6 \pi N}+4.
\end{equation}
On the other hand, if $N\kappa\geq 1/2,$ then 
\begin{equation} \label{eqn:Wootailest_alt}
\sum_{n=N+1}^\infty \left(\frac{1}{4+\kappa}\right)^{2n} 
W_{2n}^{o,o}\leq 4e^{-N\kappa/4}.
\end{equation}
Furthermore, 
\begin{equation} \label{eqn:Gooupest}
G^{o,o}\leq \frac{\log \kappa^{-1}}{\pi}+2.
\end{equation}
\end{lemma}
\noindent
Our next proposition is elementary and presumably far from 
optimal, but useful nevertheless.
\begin{proposition} \label{prop:Goxest2}
Assume that $|x|\geq 2 \kappa^{-1}$ and that 
$\kappa^{-1}\geq e^{30}.$ Then we have that 
\[
G^{o,x}\leq \frac{G^{o,o}}{2}.
\]
\end{proposition}
\begin{proof}
Assume first that $|x|$ is even. By \eqref{eqn:Goxexpress} and 
Lemma \ref{lemma:WxWo} we then have that
\[
G^{o,x}=\sum_{n=|x|}^\infty \left(\frac{1}{4+\kappa}\right)^n W_n^{o,x}
=\sum_{n=\frac{|x|}{2}}^\infty \left(\frac{1}{4+\kappa}\right)^{2n} W_{2n}^{o,x}
\leq 
\sum_{n=\frac{|x|}{2}}^\infty \left(\frac{1}{4+\kappa}\right)^{2n} W_{2n}^{o,o}.
\]
Using this, and then applying \eqref{eqn:Wootailest_alt}
with $N=|x|/2-1$ so that $N \kappa = (|x|/2-1)\kappa = (\kappa^{-1}-1)\kappa \geq 1-\kappa \geq 1/2$
(using our assumptions on $|x|$ and $\kappa^{-1}$) we have that,
\[
G^{o,x}
\leq 
\sum_{n=\frac{|x|}{2}}^\infty \left(\frac{1}{4+\kappa}\right)^{2n} 
W_{2n}^{o,o}
\leq 4 e^{-(|x|/2-1)\kappa/4}
\leq 4 e^{-(\kappa^{-1}-1)\kappa/4}\leq 4
\]
again, since we assume that $|x|\geq 2 \kappa^{-1}$ and $\kappa^{-1} \geq e^{30}$.

If instead $|x|$ is odd, we can sum over the neighbors $y \sim x$ 
of $x$ and use Lemma \ref{lemma:WxWo} to see that
\begin{eqnarray*} 
\lefteqn{G^{o,x}=\sum_{n=|x|}^\infty 
\left(\frac{1}{4+\kappa}\right)^{n} W_n^{o,x} 
=\sum_{n=|x|}^\infty 
\left(\frac{1}{4+\kappa}\right)^{n} \sum_{y \sim x}W_{n-1}^{o,y}}\\
& &= \sum_{n=|x|-1}^\infty 
\left(\frac{1}{4+\kappa}\right)^{n+1} \sum_{y\sim x}W_n^{o,y}
= \frac{1}{4+\kappa} \sum_{y\sim x}
\sum_{n=|x|-1}^\infty 
\left(\frac{1}{4+\kappa}\right)^{n} W_{n}^{o,y}\nonumber \\
& & = \frac{1}{4+\kappa} \sum_{y\sim x}
\sum_{n=\frac{|x|-1}{2}}^\infty 
\left(\frac{1}{4+\kappa}\right)^{2n} W_{2n}^{o,y}
\leq \frac{1}{4+\kappa} \sum_{y\sim x}
\sum_{n=\frac{|x|-1}{2}}^\infty 
\left(\frac{1}{4+\kappa}\right)^{2n} W_{2n}^{o,o}\nonumber \\
& & =\frac{4}{4+\kappa}
\sum_{n=\frac{|x|-1}{2}}^\infty 
\left(\frac{1}{4+\kappa}\right)^{2n} W_{2n}^{o,o}
\leq \sum_{n=\frac{|x|-1}{2}}^\infty 
\left(\frac{1}{4+\kappa}\right)^{2n} W_{2n}^{o,o}.\nonumber 
\end{eqnarray*}
Using this and \eqref{eqn:Wootailest_alt} we then see that 
\[
G^{o,x}
\leq 
\sum_{n=\frac{|x|-1}{2}}^\infty 
\left(\frac{1}{4+\kappa}\right)^{2n} W_{2n}^{o,o}
\leq 4 e^{-((|x|-1)/2-1)\kappa/4}
\leq 4 e^{-(\kappa^{1}-3/2)\kappa/4}\leq 4.
\]
Furthermore, by \eqref{eqn:Goolowest} we have that 
\[
G^{o,o}\geq 
\frac{\log \kappa^{-1}}{\pi}+1-\frac{4}{3 \pi}
\geq 10,
\]
since $\kappa^{-1}\geq e^{30}$ by assumption,
and so the statement follows.
\end{proof}

For future reference, we note that by 
\eqref{eqn:Gooupest} and \eqref{eqn:mulogG}, 
we have that 
\begin{equation}\label{eqn:muGammaoest}
\mu(\Gamma_o)=\log G^{o,o}
\leq \log \left(\frac{\log \kappa^{-1}}{\pi}+2\right).
\end{equation}
Similarly, by using \eqref{eqn:Goolowest} in place of 
\eqref{eqn:Gooupest} we have that
\begin{equation}\label{eqn:muGammaolowest}
\mu(\Gamma_o)=\log G^{o,o}
\geq \log \left(\frac{\log \kappa^{-1}}{\pi}+1-\frac{4}{3\pi}\right).
\end{equation}

Intuitively, it should be the case that for $x$ to have a ``decent 
chance of being hit'' by a walk of length $n$ starting at the 
origin $o,$ then $n$ should be of size of order close to $|x|^2$ or larger. 
Therefore, we see from \eqref{eqn:Goxexpress} that the contribution to 
$G^{o,x}$ from walks which are considerably shorter than $|x|^2$ should 
be small. This is made precise in our next lemma (again the 
proof is deferred to Appendix \ref{app:walks})
where the first statement 
\eqref{eqn:sumWoxest2} shows that the total
contribution to $G^{o,x}$ coming from walks that are shorter than 
$|x|^2/\log |x|$ is negligible as $|x|\to \infty.$ 
Both statements of this lemma will 
be useful in order to obtain a good estimate
for $G^{o,x}$ in Proposition \ref{prop:Goxest1}.

\begin{lemma} \label{lemma:sumWoxest}
For every $|x|$ large enough, we have that 
\begin{equation} \label{eqn:sumWoxest2}
\sum_{n=|x|}^{\left\lfloor \frac{|x|^2}{2\log |x|}\right\rfloor}
\left(\frac{1}{4+\kappa}\right)^{n} W_n^{o,x}
\leq 3|x|^{-1},
\end{equation}
and that
\begin{equation} \label{eqn:sumWoxest3}
\sum_{n=\left\lfloor \frac{|x|^2}{2\log |x|}\right\rfloor+1}^{|x|^2}
\left(\frac{1}{4+\kappa}\right)^{n} W_n^{o,x}
\leq 2\left(1-\frac{\kappa}{4+\kappa}\right)^{\frac{|x|^2}{2 \log |x|}}.
\end{equation}
\end{lemma}

\noindent

Using this result, we can now prove the following proposition.
\begin{proposition} \label{prop:Goxest1}
Assume that $e^9\leq \kappa^{-1}\leq |A|$. 
For large enough $|A|,$ we then have that 
\[
G^{o,x} \leq |A|^{-\frac{1}{2\mu(\Gamma_o)}}
\]
for every $x\in \BZ^2$ such that 
\[
|x|\geq |A|^{\frac{1}{\mu(\Gamma_o)}}\kappa^{-1/2}.
\]
\end{proposition}
\noindent
\begin{proof}
Using \eqref{eqn:sumWoxest2} we have that 
\begin{eqnarray}\label{eqn:Goxest1}
\lefteqn{G^{o,x}
=\sum_{n=|x|}^{\infty}\left(\frac{1}{4+\kappa}\right)^{n} W_n^{o,x}}\\
& & \leq 3|x|^{-1}
+\sum_{n=\left\lfloor \frac{|x|^2}{2\log |x|}\right\rfloor+1}^{|x|^2}
\left(\frac{1}{4+\kappa}\right)^{n} W_n^{o,x}
+\sum_{n=|x|^2+1}^\infty 
\left(\frac{1}{4+\kappa}\right)^{n} W_n^{o,x}. \nonumber 
\end{eqnarray}
Furthermore, by \eqref{eqn:muGammaoest} and 
our assumptions on $\kappa^{-1},$ 
it is easy to verify that $\mu(\Gamma_o) \leq \log \log |A|.$
Therefore, 
\begin{equation} \label{eqn:AGammaAtoinfty}
|A|^{\frac{1}{\mu(\Gamma_o)}}\geq 
|A|^{\frac{1}{\log \log |A|}}
\to \infty \qquad \text{ as } \vert A \vert \to \infty,
\end{equation}
and so we note that, by using the lower bound on $|x|,$ and that 
$\kappa^{1/2}\leq e^{9/2}$ by assumption, 
\begin{equation}\label{eqn:Goxest2}
3|x|^{-1}\leq 3|A|^{-\frac{1}{\mu(\Gamma_o)}}\kappa^{1/2}
\leq |A|^{-\frac{1}{\mu(\Gamma_o)}},
\end{equation}
for $|A|$ large enough.
Next we observe that by \eqref{eqn:sumWoxest3}
\[
\sum_{n=\left\lfloor \frac{|x|^2}{2\log |x|}\right\rfloor+1}^{|x|^2}
\left(\frac{1}{4+\kappa}\right)^{n} W_n^{o,x}
\leq 2\left(1-\frac{\kappa}{4+\kappa}\right)^{\frac{|x|^2}{2 \log |x|}}
\leq 2\left(1-\frac{\kappa}{4+\kappa}\right)
^{\frac{|A|^{\frac{2}{\mu(\Gamma_o)}}\kappa^{-1}}
{2 \log \left(|A|^{\frac{1}{\mu(\Gamma_o)}}\kappa^{-1/2}\right)}}
\]
by again using the lower bound on $|x|$ in the last inequality, together with
the fact that the function $x^2/(2\log x)$ is increasing for $x$ large.
Furthermore, by using that $\log(1-x)\leq -x$ for any $0<x<1$, we see
that 
\begin{eqnarray} \label{eqn:harke1}
\lefteqn{\exp\left(
\frac{|A|^{\frac{2}{\mu(\Gamma_o)}}\kappa^{-1}}
{2 \log \left(|A|^{\frac{1}{\mu(\Gamma_o)}}\kappa^{-1/2}\right)}
\log \left(1-\frac{\kappa}{4+\kappa}\right) \right)} \\
& & \leq \exp\left(
\frac{|A|^{\frac{2}{\mu(\Gamma_o)}}\kappa^{-1}}
{2 \log \left(|A|^{\frac{1}{\mu(\Gamma_o)}}\kappa^{-1/2}\right)}
\left(-\frac{\kappa}{5}\right) \right)
=\exp\left(-\frac{|A|^{\frac{2}{\mu(\Gamma_o)}}}
{10 \log \left(|A|^{\frac{1}{\mu(\Gamma_o)}}\kappa^{-1/2}\right)}
\right).  \nonumber
\end{eqnarray}
Using \eqref{eqn:muGammaolowest} we observe that 
$\mu(\Gamma_o)\geq 1$ and so
$|A|^{\frac{1}{\mu(\Gamma_o)}}\kappa^{-1/2}\leq |A|^{3/2}$
by our assumption on $\kappa.$ By also using 
\eqref{eqn:AGammaAtoinfty} we can therefore conclude 
from \eqref{eqn:harke1} that, for $\vert A \vert$ large enough, 
\begin{eqnarray*}
\lefteqn{\exp\left(
\frac{|A|^{\frac{2}{\mu(\Gamma_o)}}\kappa^{-1}}
{2 \log \left(|A|^{\frac{1}{\mu(\Gamma_o)}}\kappa^{-1/2}\right)}
\log \left(1-\frac{\kappa}{4+\kappa}\right) \right)}\\
& & \leq \exp\left(-\frac{|A|^{\frac{2}{\log \log |A|}}}
{10 \log |A|^{3/2}}\right)
\leq \exp\left(-2\log |A|\right)=\frac{1}{|A|^2},
\end{eqnarray*}
where the last inequality follows from elementary considerations
and is not optimal.
Therefore,
\begin{equation} \label{eqn:Goxest4}
\sum_{n=\left\lfloor \frac{|x|^2}{2\log |x|}\right\rfloor+1}^{|x|^2}
\left(\frac{1}{4+\kappa}\right)^{n} W_n^{o,x}
\leq \frac{2}{|A|^2}\leq \frac{1}{|A|},
\end{equation}
for $|A|$ large enough.

Next, assume that $|x|$ is even and note that 
\begin{eqnarray} \label{eqn:Goxest3}
\lefteqn{\sum_{n=|x|^2+1}^\infty 
\left(\frac{1}{4+\kappa}\right)^{n} W_n^{o,x} 
=\sum_{n=|x|^2+2}^\infty 
\left(\frac{1}{4+\kappa}\right)^{n} W_n^{o,x} }\\
& & = \sum_{n=|x|^2/2+1}^\infty 
\left(\frac{1}{4+\kappa}\right)^{2n} W_{2n}^{o,x}
\leq \sum_{n=|x|^2/2+1}^\infty 
\left(\frac{1}{4+\kappa}\right)^{2n} W_{2n}^{o,o}, \nonumber
\end{eqnarray}
where the first equality uses that since $|x|^2$ is even,
$W_{|x|^2+1}^{o,x}=0,$
and where the inequality follows from Lemma \ref{lemma:WxWo}.
Next, we want to apply \eqref{eqn:Wootailest_alt} with $N=|x|^2/2$
to the right hand side of \eqref{eqn:Goxest3}. For this we need to verify
that $N \kappa \geq 1/2,$ and indeed, by our assumption on $|x|,$
we here have that 
$N \kappa=|x|^2 \kappa/2 
\geq (\kappa^{-1/2}|A|^{\frac{1}{\mu(\Gamma_o)}})^2 \kappa/2
\geq |A|^{\frac{2}{\mu(\Gamma_o)}}/2\geq 1/2.$ Applying \eqref{eqn:Wootailest_alt}
we can therefore conclude that (by again using our assumption on $|x|$),
\begin{equation}\label{eqn:Goxest5}
\sum_{n=|x|^2/2+1}^\infty 
\left(\frac{1}{4+\kappa}\right)^{2n} W_{2n}^{o,o}
\leq 4e^{-|x|^2\kappa/8}
\leq 4e^{-|A|^{\frac{2}{8\mu(\Gamma_o)}}}
\leq |A|^{-2},
\end{equation}
for $|A|$ large enough.
Inserting \eqref{eqn:Goxest2}, \eqref{eqn:Goxest4} and \eqref{eqn:Goxest5}
into \eqref{eqn:Goxest1} we get that
\begin{equation} \label{eqn:Goxest6}
G^{o,x}\leq |A|^{-\frac{1}{\mu(\Gamma_o)}}
+|A|^{-1}+|A|^{-2}
\leq |A|^{-\frac{1}{2\mu(\Gamma_o)}},
\end{equation}
for $|A|$ large enough, since $\mu(\Gamma_o)\geq 1$ as before.

Assume now that $|x|$ is odd and observe that $|x|^2\geq |y|^2/2$
whenever $y\sim x$ and $\vert x \vert \geq 3$. We then 
sum over all $y \sim x$ and observe that
\begin{eqnarray*} 
\lefteqn{\sum_{n=|x|^2+1}^\infty 
\left(\frac{1}{4+\kappa}\right)^{n} W_n^{o,x} 
=\sum_{n=|x|^2+1}^\infty 
\left(\frac{1}{4+\kappa}\right)^{n} \sum_{y \sim x} W_{n-1}^{o,y}}\\
& & \leq \sum_{n=\frac{|y|^2}{2}+1}^\infty 
\left(\frac{1}{4+\kappa}\right)^{n} \sum_{y \sim x} W_{n-1}^{o,y}
=\sum_{n=\frac{|y|^2}{2}}^\infty 
\left(\frac{1}{4+\kappa}\right)^{n+1} \sum_{y \sim x} W_{n}^{o,y}
\nonumber \\
& & = \sum_{n=\frac{|y|^2}{4}}^\infty 
\left(\frac{1}{4+\kappa}\right)^{2n+1} \sum_{y \sim x} W_{2n}^{o,y}
\leq \sum_{n=\frac{|y|^2}{4}}^\infty 
\left(\frac{1}{4+\kappa}\right)^{2n+1} \sum_{y \sim x} W_{2n}^{o,o}
\nonumber \\
& & \leq \sum_{n=\frac{|y|^2}{4}}^\infty 
\left(\frac{1}{4+\kappa}\right)^{2n} W_{2n}^{o,o}
\leq 4 e^{-|y|^2 \kappa/16} \leq 4 e^{-|x|^2 \kappa/32} 
\leq 4e^{-|A|^{\frac{2}{32 \mu(\Gamma_o)}}}
\leq |A|^{-2},
\nonumber
\end{eqnarray*}
where we used \eqref{eqn:Wootailest_alt} in the fourth inequality and
that $|y|^2\geq |x|^2/2$ in the fifth (which holds for $y \sim x$ and 
$|x|\geq 4$). We see that \eqref{eqn:Goxest6}
holds also for this case, which concludes the proof.
\end{proof}

\section{Probability estimates} \label{sec:probabilities}
Recall our main goal of obtaining estimates on the cover times of a sequence
of growing sets. In order to get to that point, we need to consider a 
generic set $A,$ which one can typically think of as being very large.
Consider then
\begin{equation} \label{eqn:u*def}
u^*=\frac{\log |A|}{\mu(\Gamma_o)},
\end{equation}
and note that $u^*=u^*(|A|,\kappa).$ The relevance of $u^*$ can be seen by 
first observing that by \eqref{eqn:mulogG}
\begin{equation}\label{eqn:u*Goo}
\left(\frac{1}{G^{o,o}}\right)^{u^*}
=\exp(-u^*\mu(\Gamma_o))=|A|^{-1},
\end{equation}
and then that it follows from \eqref{eqn:probGreen}
that the expected number of uncovered vertices at time 
$u^*$ is 1. Informally, with enough independence, this is the 
intuition for why the cover time of the generic set $A$ should be 
around $u^*$ as mentioned in the discussion after the statement of
Theorem \ref{thm:main} in the Introduction. Of course, what 
constitutes enough independence is hard to quantify, and is at
the heart of the mentioned discussion as well as that of 
Section \ref{sec:examples}. 

Before we start presenting the results of this section, recall the discussion at 
the end of the introduction. In short, we want to consider the covered set at time 
$(1-\epsilon)u^*=(1-\epsilon)\frac{\log |A|}{\mu(\Gamma_o)},$ which by the 
intuition above should be ``just before coverage'' when $\epsilon>0$ is 
very small. We want to show that the 
set yet to be covered at that time consists of relatively few 
well-separated points. This result is obtained in Section \ref{sec:2ndmom}
(in particular Proposition \ref{prop:HAeps}), using a second moment argument.
In order to perform this, we need to understand the probability that two points 
$o,x$ both belong to the uncovered set. This probability will of course be 
heavily dependent on the separation of $o$ and $x,$ and the main purpose of 
this section is to understand this dependence in detail.

Our first result is the following.
\begin{proposition} \label{prop:probest}
Let $\kappa^{-1}>e^{30}$ and $\epsilon\in (0,1).$
Then, for every $x\in \BZ^2$ we have that 
\begin{equation} \label{eqn:proballdist}
\BP(\{o,x\} \cap \CC_{(1-\epsilon)u^*}=\emptyset)
\leq |A|^{-(1-\epsilon)}
\left(\frac{9}{8}\right)^{-(1-\epsilon)u^*}.
\end{equation}
Furthermore, for any $x\in \BZ^2$ such that $4\leq |x|\leq 2\kappa^{-1}$
we have that 
\begin{equation}\label{eqn:probmeddist}
\BP(\{o,x\} \cap \CC_{(1-\epsilon)u^*}=\emptyset)
\leq |A|^{-(1-\epsilon)}
\left(\frac{\log |x|}{\pi}\right)^{-(1-\epsilon)u^*}.
\end{equation}
If instead $|x|\geq 2\kappa^{-1},$ we have that
\begin{equation}\label{eqn:problargedist}
\BP(\{o,x\} \cap \CC_{(1-\epsilon)u^*}=\emptyset)
\leq |A|^{-(1-\epsilon)}
\left(\frac{\log \kappa^{-1}}{2\pi}\right)^{-(1-\epsilon)u^*}.
\end{equation}

\end{proposition}
\begin{proof}
We start with the first statement.
Use \eqref{eqn:probxyuncov} to see that
\begin{eqnarray} \label{eqn:probtoGreen1}
\lefteqn{\BP(\{o,x\} \cap \CC_{(1-\epsilon)u^*}=\emptyset)}\\
& & =\left(\frac{1}{(G^{o,o})^2-(G^{o,x})^2}\right)^{(1-\epsilon)u^*}
=\left((G^{o,o}+G^{o,x})(G^{o,o}-G^{o,x})\right)
^{-(1-\epsilon)u^*}.\nonumber
\end{eqnarray} 
Then, we have from \eqref{eqn:Greengenlest}
that $G^{o,o}-G^{o,x}\geq \frac{3}{4}$. There are now
two cases. Either $G^{o,x}\geq \frac{G^{o,o}}{2}$, in which case 
$(G^{o,o}+G^{o,x})(G^{o,o}-G^{o,x})\geq \frac{9}{8}G^{o,o}$
and so 
\[
\BP(\{o,x\} \cap \CC_{(1-\epsilon)u^*}=\emptyset)
\leq \left(G^{o,o}\frac{9}{8}\right)^{-(1-\epsilon)u^*}
=|A|^{-(1-\epsilon)} \left(\frac{9}{8}\right)^{-(1-\epsilon)u^*},
\]
by \eqref{eqn:u*Goo}, or $G^{o,x}< \frac{G^{o,o}}{2},$
in which case
$(G^{o,o}+G^{o,x})(G^{o,o}-G^{o,x})\geq G^{o,o}\frac{G^{o,o}}{2}$.
Furthermore, by \eqref{eqn:Goolowest} we have that 
$G^{o,o}\geq \frac{\log \kappa^{-1}}{\pi}\geq \frac{30}{\pi}\geq \frac{9}{4},$
by our assumption on $\kappa,$ which proves \eqref{eqn:proballdist}.

For our second statement, we note that it follows 
from \eqref{eqn:probtoGreen1} that
\begin{eqnarray} \label{eqn:probtoGreen}
\lefteqn{\BP(\{o,x\} \cap \CC_{(1-\epsilon)u^*}=\emptyset)}\\
& & \leq \left(G^{o,o}\right)^{-(1-\epsilon)u^*}
\left(G^{o,o}-G^{o,x}\right)^{-(1-\epsilon)u^*}
= |A|^{-(1-\epsilon)}\left(G^{o,o}-G^{o,x}\right)^{-(1-\epsilon)u^*},
\nonumber
\end{eqnarray}
so that by \eqref{eqn:Greenmedest}, we conclude that 
\[
\BP(\{o,x\} \cap \CC_{(1-\epsilon)u^*}=\emptyset)
\leq |A|^{-(1-\epsilon)}
\left(\frac{\log |x|}{\pi}\right)^{-(1-\epsilon)u^*},
\]
which proves \eqref{eqn:probmeddist}.

For the third statement, observe that by Proposition \ref{prop:Goxest2}
we have that $G^{o,o}-G^{o,x}\geq \frac{G^{o,o}}{2}$. Then we can use
\eqref{eqn:probtoGreen} to see that 
\begin{eqnarray*}
\lefteqn{\BP(\{o,x\} \cap \CC_{(1-\epsilon)u^*}=\emptyset)
\leq 
|A|^{-(1-\epsilon)}\left(G^{o,o}-G^{o,x}\right)^{-(1-\epsilon)u^*}}\\
& & \leq 
|A|^{-(1-\epsilon)}\left(\frac{G^{o,o}}{2}\right)^{-(1-\epsilon)u^*}
\leq |A|^{-(1-\epsilon)}
\left(\frac{\log \kappa^{-1}}{2\pi}\right)^{-(1-\epsilon)u^*}
\nonumber
\end{eqnarray*}
where we used \eqref{eqn:Goolowest} in the last inequality.
\end{proof}

Proposition \ref{prop:probest} together with Proposition 
\ref{prop:Goxest1} will suffice when proving our 
desired second moment estimates. 
However, we shall also face the issue of covering 
a relatively small number of well separated (i.e. close to the 
correlation length $\kappa^{-1/2}$) vertices.  
What we need is stated in Proposition \ref{prop:smalldistantK} below,
but in order to prove this we will first establish a preliminary
result, namely Lemma \ref{lemma:E1E2}.

For any $K\subset \BZ^2$, let 
\[
\Gamma_{K}:=\{\gamma:\gamma\cap K \neq \emptyset\}
=\bigcup_{x \in K} \Gamma_x
\]
and 
\[
\Gamma_{K_1,K_2}:=\Gamma_{K_1}\cap \Gamma_{K_2}.
\]
Recall that $\omega_u$ denotes the loop soup with intensity $u$ so that 
$\omega_u(\Gamma_{K_1,K_2})$ is the number of loops $\gamma \in\omega_u$ 
such that $\gamma\cap K_1\neq \emptyset$ and $\gamma \cap K_2 \neq \emptyset.$

\begin{lemma} \label{lemma:E1E2}
Let $K_1,K_2 \subset \BZ^2$ be disjoint, and let $E_1,E_2$ be events 
that are determined by $\omega_u$ restricted to the sets $K_1$ and $K_2$
respectively. We have that
\begin{equation} \label{eqn:E1E2diffest}
|\BP(E_1\cap E_2)
-\BP(E_1)\BP(E_2)|\leq 4\BP(\omega_u(\Gamma_{K_1, K_2})\neq 0).
\end{equation}
\end{lemma}
\noindent
\begin{proof}
Since the events $E_1,E_2$ are determined by the restrictions of 
$\omega_u$ to the 
subsets $K_1,K_2$ respectively, they are conditionally independent on 
the event that $\omega_u(\Gamma_{K_1, K_2})=0$.
We then see that,
\begin{eqnarray}
\lefteqn{\BP(E_1 \cap E_2) } \label{eq:almostindependency1}\\
& & =\BP(E_1 |\omega_u(\Gamma_{K_1, K_2}) =0)
\BP(E_2|\omega_u(\Gamma_{K_1, K_2}) =0)
\BP(\omega_u(\Gamma_{K_1, K_2}) =0) \nonumber\\
& & \ \ \ + \BP(E_1 \cap E_2| \omega_u(\Gamma_{K_1, K_2}) \neq 0) 
\BP(\omega_u(\Gamma_{K_1, K_2}) \neq 0).\nonumber
\end{eqnarray}
Furthermore, writing
\begin{eqnarray*}
\lefteqn{ \BP(E_i) 
=\BP(E_i | \omega_u(\Gamma_{K_1, K_2}) =0)\BP(\omega_u(\Gamma_{K_1, K_2}) =0)} \\
 && \ \ \ + \BP(E_i | \omega_u(\Gamma_{K_1, K_2}) \neq0)
 \BP(\omega_u(\Gamma_{K_1, K_2}) \neq0)
\end{eqnarray*}
for $i= 1,2$ and using \eqref{eq:almostindependency1}, we see that 
\begin{eqnarray*}
\lefteqn{|\BP(E_1 \cap E_2)-\BP(E_1)\BP(E_2)|}\\
& & \leq |\BP(E_1 |\omega_u(\Gamma_{K_1, K_2}) =0)
\BP(E_2|\omega_u(\Gamma_{K_1, K_2}) =0)
\BP(\omega_u(\Gamma_{K_1, K_2}) =0)-\BP(E_1)\BP(E_2)| \\
& & \hspace{10mm}+\BP(\omega_u(\Gamma_{K_1, K_2}) \neq 0) \\
& & = |\BP(E_1 |\omega_u(\Gamma_{K_1, K_2}) =0)
\BP(E_2|\omega_u(\Gamma_{K_1, K_2}) =0)
\BP(\omega_u(\Gamma_{K_1, K_2}) =0) \\
& & \hspace{10mm}-\BP(E_1)\BP(E_2 |\omega_u(\Gamma_{K_1, K_2})=0)\BP(\omega_u(\Gamma_{K_1, K_2})=0) \\
& & \hspace{10mm}+\BP(E_1)\BP(E_2| \omega_u(\Gamma_{K_1, K_2})\neq 0)\BP(\omega_u(\Gamma_{K_1, K_2})\neq 0)|
+\BP(\omega_u(\Gamma_{K_1, K_2})\neq 0) \\
& & \leq |\BP(E_1 |\omega_u(\Gamma_{K_1, K_2}) =0)
\BP(E_2|\omega_u(\Gamma_{K_1, K_2}) =0)
\BP(\omega_u(\Gamma_{K_1, K_2}) =0) \\
& & \hspace{10mm}-\BP(E_1)\BP(E_2 |\omega_u(\Gamma_{K_1, K_2})=0)\BP(\omega_u(\Gamma_{K_1, K_2})=0)|
+2\BP(\omega_u(\Gamma_{K_1, K_2})\neq 0))\\
& & \leq |\BP(E_1 |\omega_u(\Gamma_{K_1, K_2}) =0)
\BP(E_2|\omega_u(\Gamma_{K_1, K_2}) =0) \\
& & \hspace{10mm}-\BP(E_1)\BP(E_2 |\omega_u(\Gamma_{K_1, K_2})=0)|
+2\BP(\omega_u(\Gamma_{K_1, K_2})\neq 0))\\
& & =|\BP(E_1 |\omega_u(\Gamma_{K_1, K_2}) =0)
\BP(E_2|\omega_u(\Gamma_{K_1, K_2}) =0) \\
& & \hspace{10mm}-\BP(E_2 |\omega_u(\Gamma_{K_1, K_2})=0)
\BP(E_1 |\omega_u(\Gamma_{K_1, K_2})=0)\BP(\omega_u(\Gamma_{K_1, K_2})=0) \\
& & \hspace{10mm}+\BP(E_2 |\omega_u(\Gamma_{K_1, K_2})=0)\BP(E_1| \omega_u(\Gamma_{K_1, K_2})\neq 0)
\BP(\omega_u(\Gamma_{K_1, K_2})\neq 0)| \\
& & \hspace{10mm}+2\BP(\omega_u(\Gamma_{K_1, K_2})\neq 0) \\
& & \leq \BP(E_2 | \omega_u(\Gamma_{K_1, K_2})=0)
\BP(E_1 | \omega_u(\Gamma_{K_1, K_2})=0)(1-\BP(\omega_u(\Gamma_{K_1, K_2})= 0)) \\
& & \hspace{10mm}+3\BP(\omega_u(\Gamma_{K_1, K_2})\neq 0)\\
& &\leq 4\BP(\omega_u(\Gamma_{K_1, K_2})\neq 0). 
\end{eqnarray*}
\end{proof}

We are now ready to state and prove the following proposition mentioned
before Lemma \ref{lemma:E1E2}.
\begin{proposition}\label{prop:smalldistantK}
Let  $K\subset \BZ^2$ and let $\{x_1,\ldots,x_{|K|}\}$ be an enumeration 
of the vertices in $K$. Assume further that $K$ is such that
$|x_i-x_j|\geq |A|^{\frac{1}{\mu(\Gamma_o)}}\kappa^{-1/2}$ for 
every $i\neq j.$ Then we have that 
\[
|\BP(\CT(K)\leq u)-\BP(\CT(o)\leq u)^{|K|}| 
\leq 2|K|^2u |A|^{-\frac{1}{\mu(\Gamma_o)}},
\] 
whenever $u\geq 1,$ $e^9\leq \kappa^{-1}\leq |A|$ and $|A|$ is large enough.
\end{proposition}
\noindent
\begin{proof}
We start by noting that by \eqref{eqn:LeJanident}
\[
\BP(\omega_u\cap \Gamma_o \cap \Gamma_x=\emptyset)
=\left(1-\left(\frac{G^{o,x}}{G^{o,o}}\right)^2\right)^{u} 
\geq 1-u\left(\frac{G^{o,x}}{G^{o,o}}\right)^2
\]
where we used the elementary 
inequality $(1-x)^u\geq 1-ux$ for $0<x\leq 1$ and $u\geq 1,$
together with the fact that 
$\frac{G^{o,x}}{G^{o,o}}\leq 1,$ which is an immediate consequence
of \eqref{eqn:Greengenlest}.  Note that it follows from 
\eqref{eqn:muGammaolowest} and the assumption that 
$\kappa^{-1}\geq e^9$ that 
$G^{o,o}\geq \log\left(\frac{\log \kappa^{-1}}{\pi}\right)
\geq 1.$ For $u\geq 1$ we can therefore use Proposition 
\ref{prop:Goxest1} (which uses that 
$|x|\geq |A|^{\frac{1}{\mu(\Gamma_o)}}\kappa^{-1/2}$ and that $|A|$ is large enough) 
to see that
\begin{equation} \label{eqn:probuGoxxGooratio}
\BP(\omega_u\cap \Gamma_o \cap \Gamma_x \neq \emptyset) 
\leq u\left(\frac{G^{o,x}}{G^{o,o}}\right)^2
\leq u\left(G^{o,x}\right)^2\leq u |A|^{-\frac{1}{\mu(\Gamma_o)}}
\end{equation}
for every $x\in \BZ^2$ such that 
$|x|\geq |A|^{\frac{1}{\mu(\Gamma_o)}}\kappa^{-1/2}$.

We then note that for any $1\leq i\leq |K|$ and $u,\kappa$ as in 
the assumptions, we have that 
\begin{eqnarray*}
\lefteqn{\BP\left(\omega_u \cap \Gamma_{K\setminus \{x_i\}} 
\cap \Gamma_{x_i}
\neq \emptyset\right) 
\leq \sum_{x_j\in K\setminus \{x_i\}} 
\BP\left(\omega_u\cap\Gamma_{x_j}\cap \Gamma_{x_i} \neq \emptyset
\right) }\\
& & \leq (|K|-1) 
\max_{y\in K\setminus \{x_i\}}
\BP\left(\omega_u\cap \Gamma_{x_i}\cap \Gamma_{y}
\neq \emptyset\right)
\leq (|K|-1) u |A|^{-\frac{1}{\mu(\Gamma_o)}},
\end{eqnarray*}
since $|y-x_i|\geq |A|^{\frac{1}{\mu(\Gamma_o)}}\kappa^{-1/2}$ 
by assumption on $K,$
and where we used \eqref{eqn:probuGoxxGooratio} in the last inequality.
Using this and Lemma \ref{lemma:E1E2} we see that 
(with $K_1=\{x_1\}$ and $K_2=K\setminus \{x_1\}$)
\begin{eqnarray*}
\lefteqn{|\BP(\CT(K)\leq u)
-\BP(\CT(o)\leq u)\BP(\CT(K \setminus \{x_1\})\leq u)|}\\
& & =|\BP(\CT(K)\leq u)-\BP(\CT(x_1)\leq u)\BP(\CT(K \setminus \{x_1\})\leq u)| \\
& & \leq 4\BP\left(\omega_u \cap \Gamma_{K\setminus \{x_i\}} 
\cap \Gamma_{x_i}
\neq \emptyset\right)
\leq 4(|K|-1)u|A|^{-\frac{1}{\mu(\Gamma_o)}}.
\end{eqnarray*}
By iterating this we see that 
\begin{eqnarray*}
\lefteqn{|\BP(\CT(K)\leq u)-\BP(\CT(o)\leq u)^{|K|}|}\\
& & \leq 4(|K|-1)u|A|^{-\frac{1}{\mu(\Gamma_o)}}
+\BP(\CT(o)\leq u) |\BP(\CT(K\setminus \{x_1\})\leq u)
-\BP(\CT(o)\leq u)^{|K|-1}| \\
& & \leq 4(|K|-1)u|A|^{-\frac{1}{\mu(\Gamma_o)}}
+|\BP(\CT(K\setminus \{x_1\})\leq u)-\BP(\CT(o)\leq u)^{|K|-1}|\\
& & \leq \cdots \leq ((|K|-1)+(|K|-2)+\cdots+1)
4u|A|^{-\frac{1}{\mu(\Gamma_o)}} \\
& &= 4 \frac{(|K|-1)|K|}{2}u|A|^{-\frac{1}{\mu(\Gamma_o)}}
\leq 2|K|^2u|A|^{-\frac{1}{\mu(\Gamma_o)}},
\end{eqnarray*}
which concludes the proof. 
\end{proof}

\section{Second moment estimates} \label{sec:2ndmom}

For $\epsilon\in (0,1)$ we define 
\begin{equation} \label{eqn:defAeps}
A_\epsilon:=\{x\in A: x \cap\CC_{(1-\epsilon) u^*}=\emptyset \},
\end{equation}
so that $A_\epsilon$ is the set of vertices of $A$ which are uncovered at 
time $(1-\epsilon)u^*.$ By using \eqref{eqn:defAeps}, \eqref{eqn:probGreen},
\eqref{eqn:mulogG} and the definition
of $u^*$ (i.e. \eqref{eqn:u*def}) in that order, we see that for $x\in A,$
\begin{eqnarray} \label{eqn:onepoint}
\lefteqn{\BP(x\in A_\epsilon)=\BP(x \cap\CC_{(1-\epsilon) u^*}=\emptyset)
=\left(G^{o,o}\right)^{-(1-\epsilon)u^*}}\\
& & =\exp\left(-(1-\epsilon)u^* \mu(\Gamma_o)\right)
=\exp(-(1-\epsilon) \log |A|)=|A|^{-(1-\epsilon)}. \nonumber
\end{eqnarray}
Therefore,
\begin{equation}\label{eqn:expAeps}
\BE[|A_\epsilon|]=\sum_{x \in A}\BP(x\in A_\epsilon)
=|A|\exp(-(1-\epsilon)u^*\mu(\Gamma_x))
=|A|^\epsilon.
\end{equation}

In order to reach our end goal of this section, we shall need 
to establish a number 
of inequalities dealing with summing $\BP(x,y \in A_\epsilon)$ 
over various ranges of $x,y.$ We will have to consider the cases when 
the distances between $x$ and $y$ are small, 
intermediate and large separately. In addition, in order to make 
the argument work for any 
$\exp(e^{32})<\kappa^{-1}<|A|^{1-\frac{8}{\log \log |A|}}$ 
we will further have to divide the analysis into different cases depending
on the value of $\kappa^{-1}.$ In total we establish four lemmas
(Lemmas  \ref{lemma:doublesumsmalldist}, \ref{lemma:doublesummedest1},
\ref{lemma:doublesummedest2} and \ref{lemma:doublesumlargedist})
concerning such sums, and we then combine these results into Proposition 
\ref{prop:collect}. We note that not all of these results require 
equally strong conditions on $\kappa^{-1}$ and $\epsilon.$ We prefer
to write the actual conditions required in the respective statements 
of each lemma, as this makes 
it easier to see where the constraints lie.
We also note that we will actually only use the results below for 
$\epsilon$ equal to $\frac{1}{100\mu(\Gamma_o)}$ and 
$\frac{1}{400\mu(\Gamma_o)}.$ It may therefore seem superfluous to  
introduce $\epsilon$ at all, but it will make the text less technical 
in the end.

\begin{lemma} \label{lemma:doublesumsmalldist}
For any $e^9\leq \kappa^{-1}\leq |A|$ we have that 
\begin{equation}\label{eqn:doublesumsmalldist}
\sum_{x,y\in A: 1\leq |x-y| 
\leq \left(\kappa^{-1}\right)^{\frac{1}{40 \mu(\Gamma_o)}}} 
\BP(x,y \in A_\epsilon)
\leq |A|^{-\frac{1}{20\mu(\Gamma_o)}},
\end{equation}
for every $0<\epsilon\leq \frac{1}{100 \mu(\Gamma_o)}$ and $|A|$
large enough.
\end{lemma}
\begin{proof}
In order to establish 
\eqref{eqn:doublesumsmalldist}, we 
use \eqref{eqn:proballdist} and the definition of $u^*$ in 
\eqref{eqn:u*def}, together with translation invariance to see 
that 
\begin{eqnarray} \label{eqn:prelcalc}
\lefteqn{\sum_{x,y\in A: 1\leq |x-y| \leq 
\left(\kappa^{-1}\right)^{\frac{1}{40 \mu(\Gamma_o)}}} 
\BP(x,y \in A_\epsilon) \leq  \sum_{x,y\in A: 1\leq |x-y|\leq 
\left(\kappa^{-1}\right)^{\frac{1}{40 \mu(\Gamma_o)}}} 
|A|^{-(1-\epsilon)}
\left(\frac{9}{8}\right)^{-(1-\epsilon)u^*}}\\
& & = \sum_{x,y\in A: 1\leq |x-y|\leq \left(\kappa^{-1}\right)^{\frac{1}{40 \mu(\Gamma_o)}}}  
|A|^{-(1-\epsilon)}
\exp\left(-(1-\epsilon)\log\left(\frac{9}{8}\right)
\frac{\log |A|}{\mu(\Gamma_o)}\right) \nonumber \\
& & = \sum_{x,y\in A: 1\leq |x-y|\leq 
\left(\kappa^{-1}\right)^{\frac{1}{40 \mu(\Gamma_o)}}}  
|A|^{-(1-\epsilon)}
|A|^{-(1-\epsilon)\log\left(\frac{9}{8}\right)\frac{1}{\mu(\Gamma_o)}}
\nonumber \\
& & \leq 4|A|\left(\kappa^{-1}\right)^{\frac{1}{20 \mu(\Gamma_o)}}
|A|^{-(1-\epsilon)}
|A|^{-(1-\epsilon)\log\left(\frac{9}{8}\right)\frac{1}{\mu(\Gamma_o)}}
\leq 4|A|^\epsilon
\left(\kappa^{-1}\right)^{\frac{1}{20 \mu(\Gamma_o)}} 
|A|^{-\frac{1}{9\mu(\Gamma_o)}}, \nonumber
\end{eqnarray}
where we used that 
\[
(1-\epsilon)\log\left(\frac{9}{8}\right)
\geq \left(1-\frac{1}{100\mu(\Gamma_o)}\right)
\log\left(\frac{9}{8}\right)
>\frac{1}{9}
\]
since $\mu(\Gamma_o)\geq 1$ by \eqref{eqn:muGammaolowest} and the fact
that $\kappa^{-1}\geq e^9.$
Furthermore, by our assumption on $\epsilon$ we see that 
\[
4 |A|^\epsilon |A|^{-\frac{1}{9\mu(\Gamma_o)}}
\leq |A|^{-\frac{1}{10\mu(\Gamma_o)}}.
\]
We conclude that 
\begin{eqnarray*}
\lefteqn{\sum_{x,y\in A: 1\leq |x-y| \leq 
\left(\kappa^{-1}\right)^{\frac{1}{40 \mu(\Gamma_o)}}} 
\BP(x,y \in A_\epsilon)}\\
& & \leq \left(\kappa^{-1}\right)^{\frac{1}{20 \mu(\Gamma_o)}} 
|A|^{-\frac{1}{10\mu(\Gamma_o)}}
\leq |A|^{\frac{1}{20\mu(\Gamma_o)}-\frac{1}{10\mu(\Gamma_o)}}
=|A|^{-\frac{1}{20\mu(\Gamma_o)}},
\end{eqnarray*}
where we used that $\kappa^{-1}\leq |A|$ in the last inequality.
This proves \eqref{eqn:doublesumsmalldist}.
\end{proof}
\noindent
{\bf Remark:} Note that even if one replaced the upper 
bound $\left(\kappa^{-1}\right)^{\frac{1}{40 \mu(\Gamma_o)}}$ 
in the summation with $1,$ the bound would not improve much. 
In fact one would obtain 
\[
4|A|^\epsilon |A|^{-\frac{1}{9\mu(\Gamma_o)}}
\leq |A|^{-\frac{1}{10\mu(\Gamma_o)}},
\]
at the end of \eqref{eqn:prelcalc},
leading only to a slight improvement on the current bound of
$|A|^{-\frac{1}{20\mu(\Gamma_o)}}.$ In order to optimize the bound,
an improvement of \eqref{eqn:proballdist} would be required.
However, even an optimal bound in place of \eqref{eqn:proballdist}
may not fundamentally change the result.
\medskip

Our next lemma deals with intermediate scales of separation between 
$x$ and $y.$

\begin{lemma} \label{lemma:doublesummedest1}
Assume that $\exp(e^{32})\leq \kappa^{-1}\leq |A|^{}$
and that $0<\epsilon\leq \frac{1}{100 \mu(\Gamma_o)}.$ Then
for every $|A|$ large enough, 
\begin{equation}\label{eqn:doublesummeddist1}
\sum_{x,y\in A: 
\left(\kappa^{-1}\right)^{\frac{1}{40 \mu(\Gamma_o)}}\leq |x-y|
\leq \kappa^{-1/4}} 
\BP(x,y \in A_\epsilon)
\leq |A|^{-1/7}.
\end{equation}
\end{lemma}
\begin{proof}
In order to prove \eqref{eqn:doublesummeddist1}, 
we will use \eqref{eqn:probmeddist}, and therefore we observe that 
$|x-y|\leq \kappa^{-1/4}\leq 2\kappa^{-1}.$
Next, we observe that by \eqref{eqn:muGammaoest} we have that
\begin{equation} \label{eqn:mugammaosimplified}
\mu(\Gamma_o)
\leq \log \left(\frac{\log \kappa^{-1}}{\pi}+2\right)
\leq \log \log \kappa^{-1},
\end{equation}
which holds since we assume that $\kappa^{-1}\geq\exp(e^{32}).$
Therefore,
\[
|x-y|\geq \left(\kappa^{-1}\right)^{\frac{1}{40 \mu(\Gamma_o)}}
\geq \left(\kappa^{-1}\right)^{\frac{1}{40 \log \log \kappa^{-1}}}
\geq 4
\]
where the last inequality is easily checked to hold for
$\kappa^{-1} \geq \exp(e^{32})$ as in our assumption.
Hence,
the requirements for \eqref{eqn:probmeddist} are satisfied. 
We then use \eqref{eqn:u*def} and \eqref{eqn:probmeddist} to obtain that
\begin{eqnarray} \label{eqn:doublesummeddist11}
\lefteqn{\sum_{x,y\in A: 
\left(\kappa^{-1}\right)^{\frac{1}{40 \mu(\Gamma_o)}}\leq |x-y|
\leq \kappa^{-1/4}} 
\BP(x,y \in A_\epsilon)}\\
& & \leq \sum_{x,y\in A: 
\left(\kappa^{-1}\right)^{\frac{1}{40 \mu(\Gamma_o)}}\leq |x-y|
\leq \kappa^{-1/4}} 
|A|^{-(1-\epsilon)}
\left(\frac{\log |x-y|}{\pi}\right)^{-(1-\epsilon)u^*}
\nonumber \\
& & = \sum_{x,y\in A: 
\left(\kappa^{-1}\right)^{\frac{1}{40 \mu(\Gamma_o)}}\leq |x-y|
\leq \kappa^{-1/4}} 
|A|^{-(1-\epsilon)}
|A|^{-(1-\epsilon)\frac{\log \log |x-y|^{1/\pi}}{\mu(\Gamma_o)}} \nonumber\\
& & \leq  \sum_{x,y\in A: 
\left(\kappa^{-1}\right)^{\frac{1}{40 \mu(\Gamma_o)}}\leq |x-y|
\leq \kappa^{-1/4}}
|A|^{-(1-\epsilon)}
|A|^{-(1-\epsilon)
\frac{\log \log \left(\kappa^{-1}\right)
^{\frac{1}{40 \pi \mu(\Gamma_o)}}}{\mu(\Gamma_o)}} \nonumber\\
& & \leq|A| 
\left(2\kappa^{-1/4}\right)^2
|A|^{-(1-\epsilon)}
|A|^{-(1-\epsilon)
\frac{\log \log \left(\kappa^{-1}\right)
^{\frac{1}{40 \pi \mu(\Gamma_o)}}}{\mu(\Gamma_o)}}.
\nonumber
\end{eqnarray}
By again using \eqref{eqn:mugammaosimplified}, we see that
\begin{eqnarray} \label{eqn:doublesummeddist12}
\lefteqn{\frac{\log \log \left(\kappa^{-1}\right)
^{\frac{1}{40 \pi \mu(\Gamma_o)}}}{\mu(\Gamma_o)}
}\\
& &=\frac{\log \log \kappa^{-1}-\log(40 \pi)- \log(\mu(\Gamma_o))}
{\mu(\Gamma_o)}
\geq 1
-\frac{\log(40 \pi)+ \log(\mu(\Gamma_o))}{\mu(\Gamma_o)}
\geq \frac{2}{3}, \nonumber
\end{eqnarray}
where the last inequality follows since, by 
\eqref{eqn:muGammaolowest} and the fact that 
$\kappa^{-1}>\exp\left(e^{32}\right)$, we have that
\[
\mu(\Gamma_o)
\geq \log \left(\frac{\log \kappa^{-1}}{\pi}+1-\frac{4}{3\pi}\right)
\geq 30.
\]
Using \eqref{eqn:doublesummeddist11} and \eqref{eqn:doublesummeddist12} 
we see that 
\begin{eqnarray*} 
\lefteqn{\sum_{x,y\in A: 
\left(\kappa^{-1}\right)^{\frac{1}{40 \mu(\Gamma_o)}}\leq |x-y|
\leq \kappa^{-1/4}} 
\BP(x,y \in A_\epsilon)}\\
& & \leq|A| 
\left(2\kappa^{-1/4}\right)^2
|A|^{-(1-\epsilon)}
|A|^{-(1-\epsilon)
\frac{\log \log \left(\kappa^{-1}\right)
^{\frac{1}{40 \pi \mu(\Gamma_o)}}}{\mu(\Gamma_o)}} \\
& & \leq 4|A|^\epsilon \kappa^{-1/2}
|A|^{-(1-\epsilon)\frac{2}{3}}
 \leq 4|A|^{2\epsilon-\frac{2}{3}}\kappa^{-1/2}
\leq 4|A|^{2\epsilon-\frac{1}{6}} \\
& & \leq 4|A|^{\frac{2}{100 \mu(\Gamma_o)}-\frac{1}{6}} 
\leq 4|A|^{\frac{2}{3000}-\frac{1}{6}} 
\leq |A|^{-\frac{1}{7}},
\end{eqnarray*}
where we used that $\kappa^{-1}\leq |A|$ in the fourth inequality,
that $\epsilon<\frac{1}{100 \mu(\Gamma_o)}$ in the fifth inequality,
that $\mu(\Gamma_o)\geq 30$ in the penultimate inequality, and finally that 
$|A|$ is taken large enough in the last inequality.
\end{proof}

Our next lemma is an intermediate result which we will use to prove 
Lemma \ref{lemma:doublesummedest2}.
\begin{lemma} \label{lemma:kappainvupper}
For any $\kappa^{-1}$ such that
$\exp\left(e^{32}\right)\leq \kappa^{-1}
\leq |A|^{1-\frac{8}{\log \log |A|}},$
we have that 
\[
\kappa^{-1}\leq |A|^{1-\frac{6}{\mu(\Gamma_o)}},
\]
for every $|A|$ large enough.
\end{lemma}
\begin{proof}
If $\kappa^{-1} \geq |A|^{4/5}$, we can use
\eqref{eqn:muGammaolowest} to see that
\[
\mu(\Gamma_o)\geq \log \left(\frac{\log \kappa^{-1}}{\pi}\right)
\geq \log \log |A|^{4/(5\pi)}
\geq \frac{3}{4}\log \log |A|,
\]
whenever $|A|$ is large enough. Therefore we see that 
\[
|A|^{1-\frac{6}{\mu(\Gamma_o)}}
\geq |A|^{1-\frac{8}{\log \log |A|}}
\geq \kappa^{-1},
\]
as desired.

If $\exp\left(e^{32}\right)\leq \kappa^{-1}\leq |A|^{4/5}$,
\eqref{eqn:muGammaolowest} and $\kappa^{-1} \geq \exp(e^{32})$ imply that
$\mu(\Gamma_o)\geq \log \left(\frac{\log \kappa^{-1}}{\pi}\right)
\geq 30$, so that
\[
|A|^{1-\frac{6}{\mu(\Gamma_o)}}
\geq |A|^{1-\frac{6}{30}}=|A|^{4/5}\geq \kappa^{-1},
\]
which conclude the proof.
\end{proof}

\begin{lemma} \label{lemma:doublesummedest2}
Assume that $\exp(e^{32})\leq \kappa^{-1}
\leq |A|^{1-\frac{8}{\log \log |A|}}$
and that $0<\epsilon\leq \frac{1}{100 \mu(\Gamma_o)}.$ Then
for every $|A|$ large enough, 
\begin{equation}\label{eqn:doublesummeddist2}
\sum_{x,y\in A: 
\kappa^{-1/4}\leq |x-y|
\leq |A|^{\frac{1}{\mu(\Gamma_o)}}\kappa^{-1/2}} 
\BP(x,y \in A_\epsilon)
\leq |A|^{-\frac{1}{\mu(\Gamma_o)}}.
\end{equation}
\end{lemma}
\begin{proof}
In order to prove \eqref{eqn:doublesummeddist2}, 
we will again use \eqref{eqn:probmeddist}, and to that end we observe
that $4\leq \kappa^{-1/4}< 2\kappa^{-1}$ by our assumption that 
$\kappa^{-1}\geq \exp(e^{32}).$ Then, for any 
$\kappa^{-1/4}\leq |x-y| 
\leq \min\left(2 \kappa^{-1},
|A|^{\frac{1}{\mu(\Gamma_o)}}\kappa^{-1/2}\right)$ we have that 
by \eqref{eqn:probmeddist}
\begin{eqnarray} \label{eqn:doublesummeddist21}
\lefteqn{\BP(x,y \in A_\epsilon)
\leq |A|^{-(1-\epsilon)}
\left(\frac{\log |x-y|}{\pi}\right)^{-(1-\epsilon)u^*}
\leq |A|^{-(1-\epsilon)}
\left(\log \kappa^{-1/(4 \pi)}\right)^{-(1-\epsilon)u^*}}\\
& & = |A|^{-(1-\epsilon)}\exp\left(-(1-\epsilon)\frac{\log |A|}{\mu(\Gamma_o)}
\log \log \kappa^{-1/(4\pi)}\right)
=|A|^{-(1-\epsilon)}|A|^{-(1-\epsilon)
\frac{\log \log \kappa^{-1/(4\pi)}}{\mu(\Gamma_o)}}. \nonumber
\end{eqnarray}
If instead $|x-y|\geq 2 \kappa^{-1},$ we use
\eqref{eqn:problargedist} to observe that 
\[
\BP(x,y \in A_\epsilon)
\leq |A|^{-(1-\epsilon)}
\left(\frac{\log \kappa^{-1}}{2\pi}\right)^{-(1-\epsilon)u^*}
\leq |A|^{-(1-\epsilon)}
\left(\frac{\log \kappa^{-1}}{4\pi}\right)^{-(1-\epsilon)u^*}
\]
and so \eqref{eqn:doublesummeddist21} holds for every $x,y$ such 
that $\kappa^{-1/4}\leq |x-y|
\leq |A|^{\frac{1}{\mu(\Gamma_o)}}\kappa^{-1/2}.$
It follows that 
\begin{eqnarray} \label{eqn:Gammaoalpha}
\lefteqn{\sum_{x,y\in A: 
\kappa^{-1/4}\leq |x-y|
\leq |A|^{\frac{1}{\mu(\Gamma_o)}}\kappa^{-1/2}} 
\BP(x,y \in A_\epsilon)}\\
& & \leq  \sum_{x,y\in A: 
\kappa^{-1/4}\leq |x-y|
\leq |A|^{\frac{1}{\mu(\Gamma_o)}}\kappa^{-1/2}}
|A|^{-(1-\epsilon)}
|A|^{-(1-\epsilon)
\frac{\log \log \kappa^{-1/(4\pi)}}{\mu(\Gamma_o)}} \nonumber\\
& & \leq|A| 
\left(2|A|^{\frac{1}{\mu(\Gamma_o)}}\kappa^{-1/2}\right)^2
|A|^{-(1-\epsilon)}
|A|^{-(1-\epsilon)
\frac{\log \log \kappa^{-1/(4\pi)}}{\mu(\Gamma_o)}}.
\nonumber
\end{eqnarray}
As in the proof of \eqref{eqn:doublesummeddist12} we have that
\[
\frac{\log \log \kappa^{-1/(4\pi)}}{\mu(\Gamma_o)}
=\frac{\log \log \kappa^{-1}-\log(4 \pi)}{\mu(\Gamma_o)}
\geq 1-\frac{\log(4 \pi)}{\mu(\Gamma_o)}.
\]
We therefore see that 
\begin{eqnarray} \label{eqn:doublesummeddist22}
\lefteqn{\sum_{x,y\in A: 
\kappa^{-1/4}\leq |x-y|
\leq |A|^{\frac{1}{\mu(\Gamma_o)}}\kappa^{-1/2}} 
\BP(x,y \in A_\epsilon)}\\
& & \leq|A| 
\left(2|A|^{\frac{1}{\mu(\Gamma_o)}}\kappa^{-1/2}\right)^2
|A|^{-(1-\epsilon)}
|A|^{-(1-\epsilon)
\frac{\log \log \kappa^{-1/(4\pi)}}{\mu(\Gamma_o)}} \nonumber \\
& & \leq 4|A|^{\epsilon+\frac{2}{\mu(\Gamma_o)}} \kappa^{-1}
|A|^{-(1-\epsilon)\left(1-\frac{\log(4 \pi)}{\mu(\Gamma_o)}
\right)}
 \leq 4|A|^{2\epsilon+\frac{2}{\mu(\Gamma_o)}+\frac{\log(4 \pi)}{\mu(\Gamma_o)}} (\kappa|A|)^{-1} \nonumber  \\
& & \leq 4|A|^{\frac{2}{100\mu(\Gamma_o)}+\frac{2}{\mu(\Gamma_o)}+\frac{\log(4 \pi)}{\mu(\Gamma_o)}} (\kappa|A|)^{-1} 
\leq |A|^{\frac{5}{\mu(\Gamma_o)}} (\kappa|A|)^{-1}, \nonumber 
\end{eqnarray}
where we used that $\epsilon<\frac{1}{100 \mu(\Gamma_o)}$
in the penultimate inequality.
Finally, it follows from Lemma \ref{lemma:kappainvupper} 
(which uses that $\kappa^{-1}\geq \exp(e^{32})$) that
$\kappa \geq |A|^{-1+\frac{6}{\mu(\Gamma_o)}}$ and so
\[
(\kappa |A|)^{-1} 
|A|^{\frac{5}{\mu(\Gamma_o)}}
\leq 
|A|^{-\frac{6}{\mu(\Gamma_o)}}|A|^{\frac{5}{\mu(\Gamma_o)}}
\leq |A|^{-\frac{1}{\mu(\Gamma_o)}},
\]
which concludes the proof.
\end{proof}
\noindent
{\bf Remark:} As we shall see, the above lemma is the only one that
requires the upper bound on $\kappa^{-1},$ i.e. that  
$\kappa^{-1}\leq |A|^{1-\frac{8}{\log \log |A|}}.$ The other 
lemmas of this 
section only require that $\kappa^{-1}\leq |A|$ (and in addition, 
with some extra effort this bound can be relaxed). If we changed
the summation to be over the range $\kappa^{-1/4}\leq |x-y|
\leq \log |A|\kappa^{-1/2}$ (or so)  
instead of $\kappa^{-1/4}\leq |x-y|
\leq |A|^{\frac{1}{\mu(\Gamma_o)}}\kappa^{-1/2}$, then this would 
improve the bound somewhat. However, since the factor 
$|A|^{\frac{\log(4\pi)}{\mu(\Gamma_o)}}$ would still remain in the 
summation \eqref{eqn:doublesummeddist22}, this would only lead to 
a slight improvement of the upper bound of $\kappa^{-1}.$

\medskip

Our next lemma sums over pairs that are well separated.

\begin{lemma} \label{lemma:doublesumlargedist}
For any $e^9\leq \kappa^{-1} \leq |A|$ we have that 
\begin{equation}\label{eqn:doublesumlargedist}
\sum_{x,y\in A: |A|^{\frac{1}{\mu(\Gamma_o)}}\kappa^{-1/2}
\leq |x-y|} 
\BP(x,y \in A_\epsilon)
\leq |A|^{2\epsilon}\left(1+|A|^{-\frac{1}{2\mu(\Gamma_o)}}\right),
\end{equation}
whenever $0<\epsilon<1/2$ and $|A|$ is large enough.
\end{lemma}
\begin{proof}
Using \eqref{eqn:probxyuncov} we have that 
\begin{eqnarray} \label{eqn:Pxyeq}
\lefteqn{\BP(x,y \in A_\epsilon)
=\BP(o,y-x \in A_\epsilon)
=\left((G^{o,o})^2-(G^{o,y-x})^2\right)^{-(1-\epsilon)u^*}}\\
& & =(G^{o,o})^{-2(1-\epsilon)u^*}
\left(1-\left(\frac{G^{o,y-x}}{G^{o,o}}\right)^2\right)^{-(1-\epsilon)u^*}
=|A|^{-2(1-\epsilon)}
\left(1-\left(\frac{G^{o,y-x}}{G^{o,o}}\right)^2\right)^{-(1-\epsilon)u^*} \nonumber
\end{eqnarray}
where we used \eqref{eqn:u*Goo} in the last equality.
By \eqref{eqn:GoxGooratiolim}, $\frac{G^{o,y-x}}{G^{o,o}}\to 0$
as $|A|\to \infty$ since we are assuming that 
$|y-x|\geq |A|^{\frac{1}{\mu(\Gamma_o)}}\kappa^{-1/2}.$
Furthermore, $\log(1-u)\geq -2u$ whenever $0<u<1/2$ and so 
\begin{eqnarray} \label{eqn:sumest2}
\lefteqn{\left(1-\left(\frac{G^{o,y-x}}{G^{o,o}}\right)^2\right)^{-(1-\epsilon)u^*}}\\
& & =\exp\left(-(1-\epsilon)\frac{\log |A|}{\mu(\Gamma_o)}
\log\left(1-\left(\frac{G^{o,y-x}}{G^{o,o}}\right)^2\right)\right)
\nonumber \\
& & \leq \exp\left(2(1-\epsilon)\frac{\log |A|}{\mu(\Gamma_o)}
\left(\frac{G^{o,y-x}}{G^{o,o}}\right)^2\right). \nonumber
\end{eqnarray}
As before, we observe that it follows from \eqref{eqn:muGammaolowest} 
and the 
assumption that $\kappa^{-1}\geq e^9$ that both $\mu(\Gamma_o)\geq 1$
and $G^{o,o}\geq 1.$
We can now use Proposition \ref{prop:Goxest1} 
(which requires that $\kappa^{-1}\leq |A|$) to conclude that 
\begin{eqnarray} \label{eqn:sumest1}
\lefteqn{\exp\left(2(1-\epsilon)\frac{\log |A|}{\mu(\Gamma_o)}
\left(\frac{G^{o,y-x}}{G^{o,o}}\right)^2\right)
\leq 
\exp\left(2(\log |A|)
\left(G^{o,y-x}\right)^2\right)}\\
& & \leq 
\exp\left(2(\log |A|)|A|^{-\frac{1}{\mu(\Gamma_o)}}\right)
\leq \exp\left(|A|^{-\frac{2}{3\mu(\Gamma_o)}}\right)
\leq 1+|A|^{-\frac{1}{2\mu(\Gamma_o)}} \nonumber
\end{eqnarray}
where the penultimate inequality follows since 
\[
|A|^{\frac{1}{3\mu(\Gamma_o)}}
\geq |A|^{\frac{1}{3\log \log \kappa^{-1}}}
\geq |A|^{\frac{1}{3\log \log |A|}}
\geq 2(\log |A|)
\]
for large enough $|A|,$ and the last inequality follows 
since $e^x \leq 1+2x$ for small enough $x.$
Combining \eqref{eqn:Pxyeq},  \eqref{eqn:sumest2} and 
\eqref{eqn:sumest1} we see that
\begin{equation}\label{eqn:almostindependent}
\BP(x,y \in A_\epsilon)
\leq |A|^{-2(1-\epsilon)}\left(1+|A|^{-\frac{1}{2\mu(\Gamma_o)}}\right),
\end{equation}
and so 
\begin{eqnarray} 
\lefteqn{\sum_{x,y\in A: |A|^{\frac{1}{\mu(\Gamma_o)}}\kappa^{-1/2}
\leq |x-y|} 
\BP(x,y \in A_\epsilon)}\\
& & \leq \sum_{x,y\in A: 
|A|^{\frac{1}{\mu(\Gamma_o)}}\kappa^{-1/2}
\leq |x-y|}
|A|^{-2(1-\epsilon)}\left(1+|A|^{-\frac{1}{2\mu(\Gamma_o)}}\right) 
\nonumber\\
& &  \leq |A|^2|A|^{-2(1-\epsilon)}
\left(1+|A|^{-\frac{1}{2\mu(\Gamma_o)}}\right)
=|A|^{2\epsilon}\left(1+|A|^{-\frac{1}{2\mu(\Gamma_o)}}\right). \nonumber 
\end{eqnarray}
\end{proof}
\noindent
{\bf Remark:} It follows from equations \eqref{eqn:almostindependent} 
and \eqref{eqn:onepoint} that 
\[
\BP(x,y\in A_\epsilon)
\leq |A|^{-2(1-\epsilon)}+R
=\BP(x\in A_\epsilon)^2+R
\]
where $R$ is some small error term. 
Morally, this means that $o,x$ are ``almost''
independently covered. 
This is not surprising considering that they are separated by 
a distance
close to the diameter of a typical loop, i.e.\ $\kappa^{-1/2}$ 
(recall the discussion after the statement of Theorem 
\ref{thm:main} in the Introduction).

\medskip

We collect the above lemmas in the following proposition.
\begin{proposition} \label{prop:collect}
For any 
$\exp(e^{32})\leq \kappa^{-1}\leq |A|^{1-\frac{8}{\log \log |A|}},$ 
$0<\epsilon\leq \frac{1}{100 \mu(\Gamma_o)}$ and $|A|$ large enough
we have that 
\begin{equation} \label{eqn:propcollect1}
\sum_{x,y\in A: 0<|x-y|\leq |A|^{\frac{1}{\mu(\Gamma_o)}}\kappa^{-1/2}}
\BP(x,y\in A_\epsilon)\leq 
2|A|^{-\frac{1}{20\mu(\Gamma_o)}},
\end{equation}
and that 
\begin{equation} \label{eqn:propcollect2}
\sum_{x,y\in A: |x-y|>0}
\BP(x,y\in A_\epsilon)
\leq |A|^{2\epsilon}\left(1+3|A|^{-\frac{1}{20\mu(\Gamma_o)}}\right).
\end{equation}
\end{proposition}
\begin{proof}
We start by considering \eqref{eqn:propcollect1}. Since we assume 
that $\exp(e^{32})\leq \kappa^{-1}\leq |A|^{1-\frac{8}{\log \log |A|}},$ we 
can use 
\eqref{eqn:doublesumsmalldist}, \eqref{eqn:doublesummeddist1} and 
\eqref{eqn:doublesummeddist2} to see that 
\begin{eqnarray*}
\lefteqn{\sum_{x,y\in A: 0<|x-y|\leq |A|^{\frac{1}{\mu(\Gamma_o)}}\kappa^{-1/2}}
\BP(x,y\in A_\epsilon)}\\
& & \leq 
\sum_{x,y\in A: 1\leq |x-y| 
\leq \left(\kappa^{-1}\right)^{\frac{1}{40 \mu(\Gamma_o)}}} 
\BP(x,y \in A_\epsilon)
+\sum_{x,y\in A: 
\left(\kappa^{-1}\right)^{\frac{1}{40 \mu(\Gamma_o)}}\leq |x-y|
\leq \kappa^{-1/4}} 
\BP(x,y \in A_\epsilon) \\
& & \hspace{20mm} 
+\sum_{x,y\in A: 
\kappa^{-1/4}\leq |x-y|
\leq |A|^{\frac{1}{\mu(\Gamma_o)}}\kappa^{-1/2}} 
\BP(x,y \in A_\epsilon) \\
& & \leq 
|A|^{-\frac{1}{20\mu(\Gamma_o)}}+|A|^{-1/7}
+|A|^{-\frac{1}{\mu(\Gamma_o)}}
\leq 2|A|^{-\frac{1}{20\mu(\Gamma_o)}}
\end{eqnarray*}
for all $|A|$ large enough (since $\mu(\Gamma_o)\geq 30$ by our 
assumption on $\kappa^{-1}$ and \eqref{eqn:muGammaolowest}).

The second statement follows by using \eqref{eqn:doublesumlargedist}
together with \eqref{eqn:propcollect1} and observing that 
\[
2|A|^{-\frac{1}{20\mu(\Gamma_o)}}
+|A|^{2\epsilon}\left(1+|A|^{-\frac{1}{2\mu(\Gamma_o)}}\right)
\leq |A|^{2\epsilon}\left(1+3|A|^{-\frac{1}{20\mu(\Gamma_o)}}\right)
\]
for $|A|$ large enough.
\end{proof}

\medskip

We shall now use Proposition \ref{prop:collect} to prove that 
the uncovered region at time $(1-\epsilon)u^*$ consists of 
a small collection of vertices
all separated by a large distance.
To that end, define, for $0<\epsilon<1$, 
\begin{eqnarray} \label{eqn:defHAeps}
\lefteqn{H_{A,\epsilon}:=\Big{\{}K \subset A: ||K|-|A|^\epsilon| 
\leq |A|^{3\epsilon/4},}\\
& & \hspace{10mm}
\textrm{ and } |x-y|\geq \kappa^{-1/2}|A|^{\frac{1}{\mu(\Gamma_o)}}
 \textrm{ for every distinct } x,y\in K \Big{\}}. \nonumber
\end{eqnarray}

\begin{proposition} \label{prop:HAeps}
For any 
$\exp(e^{32})<\kappa^{-1}\leq |A|^{1-\frac{8}{\log \log |A|}}$ and 
$0<\epsilon\leq \frac{1}{100 \mu(\Gamma_o)}$
we have that 
\[
\BP(A_\epsilon \not \in H_{A,\epsilon})
\leq 3|A|^{-\epsilon/2}
\]
for every $|A|$ large enough. 
\end{proposition}
\noindent
\begin{proof}
We use \eqref{eqn:propcollect1} of Proposition \ref{prop:collect}
to see that 
\begin{eqnarray} \label{eqn:HAepsest1}
\lefteqn{\BP\left(\exists x,y\in A_{\epsilon}: 
0<|x-y|<\kappa^{-1/2}|A|^{\frac{1}{\mu(\Gamma_o)}}\right)}\\
& & 
\leq \sum_{x,y\in A: 0<|x-y|\leq 
\kappa^{-1/2}|A|^{\frac{1}{\mu(\Gamma_o)}}} 
\BP(x,y\in A_\epsilon) 
\leq 2|A|^{-\frac{1}{20\mu(\Gamma_o)}} \nonumber
\end{eqnarray}
for every $|A|$ large enough.
Next we observe that by \eqref{eqn:propcollect2} of 
Proposition \ref{prop:collect} we have that 
\begin{eqnarray*}
\lefteqn{\BE[|A_\epsilon|^2]
=\sum_{x,y\in A} \BP(x,y\in A_\epsilon)}\\
& & =\sum_{x\in A} \BP(x\in A_\epsilon)
+\sum_{x,y\in A: |x-y|>0} \BP(x,y\in A_\epsilon)
\leq |A|^\epsilon 
+|A|^{2\epsilon}\left(1+3|A|^{-\frac{1}{20\mu(\Gamma_o)}}\right).
\end{eqnarray*}
Recall \eqref{eqn:expAeps} which states that 
$\BE[|A_\epsilon|]=|A|^\epsilon,$ so that
$\BE[(|A_\epsilon|-|A|^{\epsilon})^2]
=\BE[|A_\epsilon|^2]-|A|^{2 \epsilon}.$ 
Therefore, by Chebyshev's inequality,
\begin{eqnarray*} 
\lefteqn{\BP\left(||A_\epsilon|-|A|^\epsilon| 
\geq |A|^{3\epsilon/4}\right)
\leq \frac{\BE[|A_\epsilon|^2]-|A|^{2 \epsilon}}
{|A|^{3\epsilon/2}}}\\
& &\leq \frac{|A|^\epsilon 
+|A|^{2\epsilon}\left(1+3|A|^{-\frac{1}{20\mu(\Gamma_o)}}\right)
-|A|^{2 \epsilon}}
{|A|^{3\epsilon/2}} \\
& &  =|A|^{-\epsilon/2}
+3|A|^{\epsilon/2}|A|^{-\frac{1}{20\mu(\Gamma_o)}} \\
& & \leq|A|^{-\epsilon/2}
+3|A|^{\epsilon/2}|A|^{-5 \epsilon}
\leq 2|A|^{-\epsilon/2}
\nonumber 
\end{eqnarray*}
for $|A|$ large enough by using our assumption that 
$\epsilon\leq \frac{1}{100\mu(\Gamma_o)}$ in the penultimate 
inequality. Thus,
\begin{equation} \label{eqn:HAepsest2}
\BP\left(||A_\epsilon|-|A|^\epsilon| 
\geq |A|^{3\epsilon/4}
\right)
\leq 2|A|^{-\epsilon/2}. 
\end{equation}
Combining \eqref{eqn:HAepsest1} and \eqref{eqn:HAepsest2} we then 
find that 
\[
\BP(A_\epsilon \not \in H_{A,\epsilon})
\leq 2|A|^{-\frac{1}{20\mu(\Gamma_o)}}
+2|A|^{-\epsilon/2}
\leq  3|A|^{-\epsilon/2},
\]
where we again used the upper bound on $\epsilon.$
\end{proof}

\section{Proof of main theorem} \label{sec:proofofmain}

We will now put all the pieces together.

\noindent
\begin{proof}[Proof of Theorem \ref{thm:main}.]
In this proof we will fix $\epsilon$ to be 
equal to $\frac{1}{100 \mu(\Gamma_o)},$ but we shall apply 
Proposition \ref{prop:HAeps} to $\epsilon$ and $\epsilon/4.$
Our aim is to establish that 
\begin{equation}\label{eqn:basicineq2}
\sup_{z \in \BR}|\BP(\mu(\Gamma_o)\CT(A)-\log |A|\leq z)
-\exp(-e^{-z})|
\leq 
12 |A|^{-\epsilon/8}
=12|A|^{-\frac{1}{800 \mu(\Gamma_o)}}
\end{equation}
for every large enough finite set $A\subset \BZ^2.$

\medskip

We will divide the proof of \eqref{eqn:basicineq2} into three cases, depending on the value of $z.$

\noindent
{\bf Case 1:} Here, $z \leq -\frac{\epsilon}{4} \log |A|$. 
Then
\begin{eqnarray*}
\lefteqn{\BP(\mu(\Gamma_o)\CT(A)-\log |A|\leq z)
=\BP\left(\CT(A) \leq u^* + \frac{z}{\mu(\Gamma_o)}\right) }\\
&& \leq \BP\left(\CT(A)\leq  \left(u^*- \frac{\epsilon}{4} u^*\right)\right)
=\BP\left(\left\{x \in A : \CT(x) > \left(1 - \frac{\epsilon}{4} \right)u^*\right\} 
= \emptyset \right) 
= \BP(A_{\epsilon/4} = \emptyset),
\end{eqnarray*}
by the definition of $A_\epsilon$ in \eqref{eqn:defAeps}. Moreover, 
$\BP(A_{\epsilon/4} = \emptyset) 
\leq \BP(A_{\epsilon/4}\notin H_{A,\epsilon/4})$, 
since $H_{A,\epsilon/4}$ does not contain the empty set.
It then follows from Proposition \ref{prop:HAeps} that for $|A|$ 
large enough,
\[
\BP(\mu(\Gamma_o)\CT(A)-\log |A|\leq z)
\leq \BP(A_{\epsilon/4} \notin H_{A,\epsilon/4})
 \leq 3|A|^{-\epsilon/8}.
\]
Next, using that $z\leq  -\frac{\epsilon}{4} \log |A|$ it follows that 
$e^{-z}\geq |A|^{\epsilon/4},$ and so, for all $|A|$ large enough, 
\[
\exp(-e^{-z}) \leq \exp\left(-|A|^{\epsilon/4}\right)
\leq |A|^{-\epsilon/4}.
\]
Therefore, 
\begin{eqnarray}
\lefteqn{|\BP(\mu(\Gamma_o)\CT(A)-\log |A|\leq z) -\exp(-e^{-z})|}\\
& & \leq \BP(\mu(\Gamma_o)\CT(A)-\log |A|\leq z) + \exp(-e^{-z})
\leq 3|A|^{-\epsilon/8}+|A|^{-\epsilon/4}
\leq 4|A|^{-\epsilon/8},
\nonumber
\end{eqnarray}
and so for all $|A|$ large enough, 
\eqref{eqn:basicineq2} is satisfied in this case.

\medskip
\noindent
{\bf Case 2:} Assume now instead that 
$z \geq \log |A|$. Then,
\begin{eqnarray*}
\lefteqn{\BP(\mu(\Gamma_o)\CT(A)-\log |A|>z)
=\BP\left(\CT(A) > u^* + \frac{z}{\mu(\Gamma_o)}\right)} \\
& & \leq  \BP(\CT(A) > 2u^*)
= \BP\left(\bigcup_{x \in A} 
 \left\{\CT(x)>2u^*\right\} \right)
  \\
 & & \leq |A|\BP(\CT(o) > 2u^*)
 = |A|\exp(-2u^* \mu(\Gamma_o)) = |A|^{-1}.
\end{eqnarray*}
Then, since $z \geq  \log |A|$, we have that 
$e^{-z}\leq e^{-\log |A|}=|A|^{-1}$ and so 
\[
\exp(-e^{-z}) \geq \exp(-|A|^{-1}) \geq 1-|A|^{-1},
\] 
since $e^x \geq 1 + x$ for every $x$. 
This, and the above equation gives
\begin{eqnarray} \label{eq:case2bound}
\lefteqn{|\BP(\mu(\Gamma_o)\CT(A)-\log |A| \leq z) -\exp(-e^{-z})|}\\
&&=|1-\BP(\mu(\Gamma_o)\CT(A)-\log |A|>z) -\exp(-e^{-z})| \nonumber\\
&&\leq \BP(\mu(\Gamma_o)(\CT(A)-\log |A|)>z) + |1 - \exp(- e^{-z})| 
\leq |A|^{-1} + |A|^{-1} , \nonumber
\end{eqnarray}
and so \eqref{eqn:basicineq2} holds also in this case.

\medskip
\noindent
{\bf Case 3:} Assume that 
$z \in \left(-\frac{\epsilon}{4} \log |A|,\log |A|\right)$ 
and start by observing that
\begin{eqnarray} \label{eq:firstterm}
\lefteqn{|\BP(\mu(\Gamma_o)\CT(A)-\log |A|\leq z) 
- \exp(-e^{-z})|} \\
& & \leq |\BP(\mu(\Gamma_o)\CT(A)-\log |A|\leq z) 
- \BP(\mu(\Gamma_o)\CT(A)-\log |A|\leq z, A_\epsilon \in H_{A,\epsilon})| \nonumber \\
& & \hspace{4mm} + |\exp(-e^{-z})\BP(A_\epsilon \in H_{A,\epsilon})
-\exp(-e^{-z})| \nonumber \\
& & \hspace{4mm}
+|\BP(\mu(\Gamma_o)\CT(A)-\log |A|\leq z, A_\epsilon \in H_{A,\epsilon})
-\exp(-e^{-z})\BP(A_\epsilon \in H_{A,\epsilon})| . \nonumber
\end{eqnarray}
We will now consider the three terms on the right hand side separately.

For the first term we note that  
\begin{eqnarray} \label{eqn:term1}
\lefteqn{|\BP(\mu(\Gamma_o)\CT(A)-\log |A|\leq z) 
-\BP(\mu(\Gamma_o)\CT(A)-\log |A|\leq z,A_\epsilon \in H_{A,\epsilon})|}
\\
&& = \BP(\mu(\Gamma_o)\CT(A)-\log |A|\leq z, 
A_\epsilon \not \in H_{A,\epsilon}) 
 \leq \BP(A_\epsilon \notin H_{A,\epsilon}) 
\leq 3|A|^{-\epsilon/2}\nonumber
\end{eqnarray}
by Proposition \ref{prop:HAeps}.
Similarly, for the second term of the right hand side of \eqref{eq:firstterm}, 
we observe that 
\begin{equation} \label{eqn:term2}
|\exp(-e^{-z})\BP(A_\epsilon \in H_{A,\epsilon})-\exp(-e^{-z})|
=\exp(-e^{-z}) \BP(A_\epsilon \notin H_{A,\epsilon}) 
\leq  3|A|^{-\epsilon/2}
\end{equation}
again by Proposition \ref{prop:HAeps}.

Consider now the third and final term of \eqref{eq:firstterm}. 
Let $K \in H_{A,\epsilon}$. We will show below that
\begin{equation} \label{eqn:Tsrhoprel}
|\BP(\mu(\Gamma_o)\CT(A)-\log |A|\leq z | A_\epsilon=K)
-\exp(-e^{-z})| \leq 5|A|^{-\epsilon/4}.
\end{equation}
After multiplication by $\BP(A_\epsilon = K)$ and summation over all 
$K \in H_{A,\epsilon}$ we then obtain
\begin{equation}\label{eqn:term3}
|\BP(\mu(\Gamma_o)\CT(A)-\log |A|\leq z, A_\epsilon \in H_{A,\epsilon})
-\exp(-e^{-z})\BP(A_\epsilon \in H_{A,\epsilon})| 
\leq 5|A|^{-\epsilon/4}.
\end{equation}
Summing the contributions from \eqref{eqn:term1}, \eqref{eqn:term2}
and \eqref{eqn:term3} 
we conclude from \eqref{eq:firstterm} that 
\begin{equation}\label{eqn:case3bound}
|\BP(\mu(\Gamma_o)(\CT(A)-\log |A|)\leq z)
-\exp(-e^{-z})|
 \leq 6|A|^{-\epsilon/2}+5|A|^{-\epsilon/4}\leq 12|A|^{-\epsilon/4}
\end{equation}
for all $z \in (-\epsilon \log |A|, \log |A|)$ and $|A|$ large enough. 
It may be worth recalling that, from the start of the proof, we
assume that $\epsilon=\frac{1}{100\mu(\Gamma_o)}.$
This then proves \eqref{eqn:basicineq2} and completes the proof, modulo \eqref{eqn:Tsrhoprel}.

In order to prove \eqref{eqn:Tsrhoprel} we consider the conditional probability
\[
\BP(\mu(\Gamma_o)\CT(A)-\log |A|\leq z | A_\epsilon=K)
=\frac{\BP\left(\CT(A)\leq u^*+\frac{z}{\mu(\Gamma_o)}, A_\epsilon=K\right)}
{\BP(A_\epsilon=K)}.
\] 
Let $\omega_{u_1,u_2}$ denote the loops arriving between times 
$u_1$ and $u_2$ where $u_1<u_2.$ On the event that $A_\epsilon=K$ 
it must be that 
$K$ is covered by the loops arriving between times $(1-\epsilon)u^*$ and 
$u^*+\frac{z}{\mu(\Gamma_o)}$ for the event 
$\CT(A)\leq u^*+\frac{z}{\mu(\Gamma_o)}$ to also occur. 
Therefore,
\begin{eqnarray*}
\lefteqn{\BP\left(\CT(A)\leq u^*+\frac{z}{\mu(\Gamma_o)}, 
A_\epsilon=K\right)
 =\BP\left((1-\epsilon)u^*\leq 
\CT(K)\leq u^*+\frac{z}{\mu(\Gamma_o)}, A_\epsilon=K\right)}\\
& & =\BP\left(K\subset 
\bigcup_{\gamma\in \omega_{(1-\epsilon)u^*,u^*+\frac{z}{\mu(\Gamma_o)}}} \gamma, A_\epsilon=K\right)
=\BP\left(K\subset 
\bigcup_{\gamma\in \omega_{(1-\epsilon)u^*,u^*+\frac{z}{\mu(\Gamma_o)}}} \gamma\right)
\BP( A_\epsilon=K)\\
& & =\BP\left(\CT(K)\leq \epsilon u^*+\frac{z}{\mu(\Gamma_o)}\right)
\BP( A_\epsilon=K),
\end{eqnarray*}
where the last equality follows from the Poissonian nature of the 
loop process, which implies that the distribution of 
the loops that fall between times $(1-\epsilon)u^*$ and 
$u^*+\frac{z}{\mu(\Gamma_o)}$
is simply a Poissonian loop process with intensity 
$u^*+\frac{z}{\mu(\Gamma_o)}-(1-\epsilon)u^*
=\epsilon u^*+\frac{z}{\mu(\Gamma_o)}.$

We therefore see that 
\begin{equation} \label{eq:firstbound} 
\BP\left(\CT(A)\leq u^*+\frac{z}{\mu(\Gamma_o)} \Big{|} A_\epsilon = K\right)
= \BP\left(\CT(K) \leq \epsilon u^*+\frac{z}{\mu(\Gamma_o)}\right),
\end{equation}
and using \eqref{eq:firstbound} we have
\begin{eqnarray} \label{eqn:Tsrhoprel2}
\lefteqn{\left|\BP\left(\CT(A)\leq u^*+\frac{z}{\mu(\Gamma_o)} 
\Big{|} A_\epsilon = K\right)  
- \exp(-e^{-z})\right|} \\
&&= \left|\BP\left(\CT(K) \leq \epsilon u^*+\frac{z}{\mu(\Gamma_o)}\right)
-\exp(-e^{-z})\right| \nonumber \\
&&\leq \left|\BP\left(\CT(K) \leq \epsilon u^*+\frac{z}{\mu(\Gamma_o)}\right)
 -\BP\left(\CT(o) \leq \epsilon u^*+\frac{z}{\mu(\Gamma_o)}\right)^{|K|} \right|
\nonumber \\
&& \hspace{4mm} 
+\left|\BP\left(\CT(o) \leq \epsilon u^*+\frac{z}{\mu(\Gamma_o)}\right)^{|K|}
-\exp(-e^{-z})\right|. \nonumber 
\end{eqnarray}
We will deal with the two terms on the right hand side of 
\eqref{eqn:Tsrhoprel2} separately.

For the first term, we will use Proposition \ref{prop:smalldistantK}.
Let therefore $x,y\in K$ be distinct. By the definition of $H_{A,\epsilon}$ in \eqref{eqn:defHAeps},
we have that, if $x,y \in K$ and $K \in H_{A,\epsilon}$, then
$|x-y|\geq \kappa^{-1/2}|A|^{\frac{1}{\mu(\Gamma_o)}}.$
Furthermore, if $K \in H_{A,\epsilon}$, then  
$||K|-|A|^\epsilon| \leq |A|^{3\epsilon/4},$
and so we have that  
\begin{equation}\label{eq:intermediate1}
|A|^\epsilon-|A|^{3\epsilon/4}
\leq |K| \leq |A|^\epsilon+|A|^{3\epsilon/4}
\end{equation}
and in particular, \eqref{eq:intermediate1} implies that 
$|K|\leq 2|A|^\epsilon.$

We now want to apply Proposition \ref{prop:smalldistantK}, and to 
that end we note that by \eqref{eqn:muGammaoest} we have that
$\mu(\Gamma_o)\leq \log \log \kappa^{-1}\leq \log \log |A|.$
Therefore, for every 
$z\in \left(-\frac{\epsilon}{4}\log |A|, \log |A|\right)$ we have that
\begin{eqnarray*}
\lefteqn{\epsilon u^*+\frac{z}{\mu(\Gamma_o)}
\geq \epsilon \frac{\log |A|}{\mu(\Gamma_o)}-
\epsilon\frac{\log |A|}{4\mu(\Gamma_o)}
=\epsilon\frac{3\log |A|}{4\mu(\Gamma_o)}}\\
& & =\frac{3}{400} \frac{\log |A|}{(\mu(\Gamma_o))^2}
\geq \frac{3}{400} \frac{\log |A|}{(\log \log |A|)^2}\geq 1,
\end{eqnarray*}
whenever $|A|$ is large enough. We can therefore use
Proposition \ref{prop:smalldistantK} together with
$\vert K \vert \leq 2 \vert A \vert^{\epsilon}$ (which follows from \eqref{eq:intermediate1}),
$\epsilon \leq 1$, and the fact that 
$\frac{z}{\mu(\Gamma_o)}\leq \frac{\log |A|}{\mu(\Gamma_o)}
=u^*$ (this is the only place where we use the upper bound on $z$),
to see that 
\begin{eqnarray*}
\lefteqn{\left|\BP\left(\CT(K) \leq \epsilon u^*+\frac{z}{\mu(\Gamma_o)}\right)
 -\BP\left(\CT(o) \leq \epsilon u^*+\frac{z}{\mu(\Gamma_o)}\right)^{|K|} \right|}\\
& & \leq 2|K|^2 \left(\epsilon u^*+\frac{z}{\mu(\Gamma_o)}\right)
|A|^{-\frac{1}{\mu(\Gamma_o)}}
\leq 16|A|^{2\epsilon}u^*
|A|^{-\frac{1}{\mu(\Gamma_o)}}.
\end{eqnarray*}
Next, observe that for any $\kappa^{-1}\leq |A|$ we can use 
\eqref{eqn:mugammaosimplified} to see that
\[
|A|^{\frac{1}{2\mu(\Gamma_o)}}
\geq |A|^{\frac{1}{2\log \log \kappa^{-1}}}
\geq |A|^{\frac{1}{2\log \log |A|}}
\geq \log |A|,
\]
for any $|A|$ large enough. Therefore, we see that since 
$\mu(\Gamma_o)\geq 1,$
\[
16|A|^{2\epsilon}u^*|A|^{-\frac{1}{\mu(\Gamma_o)}}
= 16|A|^{2\epsilon}\frac{\log |A|}{\mu(\Gamma_o)}
|A|^{-\frac{1}{\mu(\Gamma_o)}}
\leq 16|A|^{2\epsilon}|A|^{-\frac{1}{2\mu(\Gamma_o)}} \nonumber
\leq |A|^{-\epsilon},
\]
for $|A|$ large enough by using the assumption on $\epsilon.$
We therefore conclude that 
\begin{equation} \label{eqn:independentest}
\left|\BP\left(\CT(K) \leq \epsilon u^*+\frac{z}{\mu(\Gamma_o)}\right)
 -\BP\left(\CT(o) \leq \epsilon u^*+\frac{z}{\mu(\Gamma_o)}\right)^{|K|} \right|
 \leq |A|^{-\epsilon}
\end{equation}
for $|A|$ large enough.

We can now turn to the second term of the right hand side of
\eqref{eqn:Tsrhoprel2}. As before,
$\exp\left(-\epsilon u^*\right)
=\exp\left(-\epsilon \frac{\log |A|}{\mu(\Gamma_o)}\right)$ and so
we have that 
\begin{eqnarray*}
\lefteqn{
\BP\left(\CT(o) \leq \epsilon u^*+\frac{z}{\mu(\Gamma_o)}\right)^{|K|}}\\
& &  = \left(1 - \exp\left(\left(-\epsilon u^*-\frac{z}{\mu(\Gamma_o)}\right)
\mu(\Gamma_o)\right)\right)^{|K|} = \left(1 
- \frac{e^{-z}}{|A|^\epsilon}\right)^{|K|}.
\end{eqnarray*}
Then by \eqref{eq:intermediate1}, 
\begin{equation}\label{eq:thirdbound}
\left(1- \frac{e^{-z}}{|A|^\epsilon}\right)
^{|A|^\epsilon+|A|^{3\epsilon/4}}
\leq 
\BP\left(\CT(o) \leq \epsilon u^*+\frac{z}{\mu(\Gamma_o)}\right)^{|K|} 
\leq \left(1-\frac{e^{-z}}{|A|^\epsilon}\right)
^{|A|^\epsilon-|A|^{3\epsilon/4}}.
\end{equation}
Therefore, using that $\log (1-x)\geq -x-x^2 $ 
for every $0<x<1/2,$ we get that 
\begin{eqnarray} \label{eqn:expPest1}
\lefteqn{\exp(-e^{-z})-\BP\left(\CT(o) \leq \epsilon u^*+\frac{z}{\mu(\Gamma_o)}\right)^{|K|}}\\
& & \leq \exp(-e^{-z})-
\left(1-\frac{e^{-z}}{|A|^\epsilon}\right)
^{|A|^\epsilon+|A|^{3\epsilon/4}}
\nonumber\\
& & =\exp(-e^{-z})
-\exp\left(\left(|A|^\epsilon+|A|^{3\epsilon/4}\right)
\log\left(1-\frac{e^{-z}}{|A|^\epsilon}\right)\right) 
\nonumber\\
& & \leq \exp(-e^{-z})
-\exp\left(\left(|A|^\epsilon+|A|^{3\epsilon/4}\right)
\left(-\frac{e^{-z}}{|A|^\epsilon}-\frac{e^{-2z}}{|A|^{2\epsilon}}
\right) \right) 
\nonumber\\
& & = \exp(-e^{-z})
-\exp\left(-e^{-z}-e^{-z}|A|^{-\epsilon/4}
-e^{-2z}|A|^{-\epsilon}\left(1 + |A|^{-\epsilon/4}\right)\right) 
\nonumber\\
& & =\exp(-e^{-z})\left(1-\exp\left(-e^{-z}|A|^{-\epsilon/4}
-e^{-2z}|A|^{-\epsilon}\left(1 + |A|^{-\epsilon/4}\right)\right)\right)
\nonumber\\
& & \leq \exp(-e^{-z})\left(e^{-z}|A|^{-\epsilon/4} 
+e^{-2z}|A|^{-\epsilon}\left(1 + |A|^{-\epsilon/4}\right)\right),
\nonumber
\end{eqnarray}
where we used that $1-e^{-x}\leq x$ for $x>0$ in the last inequality.
It is easy to check that $ye^{-y}\leq 1$ for every $y,$ and so 
\begin{equation} \label{eqn:expPest2}
\exp(-e^{-z})e^{-z}|A|^{-\epsilon/4}
\leq |A|^{-\epsilon/4}.
\end{equation}
Furthermore, 
\begin{equation} \label{eqn:expPest3}
\exp(-e^{-z})\left(
e^{-2z}|A|^{-\epsilon}\left(1 + |A|^{-\epsilon/4}\right)\right)
\leq 2e^{-2z}|A|^{-\epsilon}\leq 2|A|^{-\epsilon/2},
\end{equation}
since $z\geq -\frac{\epsilon}{4}\log |A|.$ 
Combining 
\eqref{eqn:expPest1}, \eqref{eqn:expPest2} and \eqref{eqn:expPest3} 
we obtain 
\begin{equation}\label{eqn:expPest4}
\exp(-e^{-z})
-\BP\left(\CT(o) \leq \epsilon u^*+\frac{z}{\mu(\Gamma_o)}\right)^{|K|}
\leq |A|^{-\epsilon/4}
+2|A|^{-\epsilon/2}
\leq 3|A|^{-\epsilon/4}.
\end{equation}
Similarly, we have that  $\log(1-x)\leq -x$ 
for every $0<x<1/2$ and therefore
\begin{eqnarray} \label{eqn:expPest5}
\lefteqn{\exp(-e^{-z})-\BP\left(\CT(o) \leq \epsilon u^*+\frac{z}{\mu(\Gamma_o)}\right)^{|K|}}\\
& & \geq \exp(-e^{-z})-\left(1-\frac{e^{-z}}{|A|^\epsilon}\right)
^{|A|^\epsilon\left(1-
|A|^{-\epsilon/4}\right)} \nonumber \\
& & =\exp(-e^{-z})
-\exp\left(|A|^\epsilon\left(1 - 
|A|^{-\epsilon/4}\right)
\log\left(1-\frac{e^{-z}}{|A|^\epsilon}\right)\right) 
\nonumber\\
& & \geq \exp(-e^{-z})
-\exp\left(|A|^\epsilon\left(1 - |A|^{-\epsilon/4}\right)
\left(-\frac{e^{-z}}{|A|^\epsilon}\right) \right) 
\nonumber\\
& & = \exp(-e^{-z})
-\exp\left(-e^{-z}\left(1 - |A|^{-\epsilon/4}\right) \right).
\nonumber
\end{eqnarray}
It is easy to check that the function $f(x)=x-x^{1-\beta}$ where 
$0<\beta\leq 1,$ is minimized when $x=(1-\beta)^{1/\beta}$ so that 
\[
f(x)\geq (1-\beta)^{1/\beta}-(1-\beta)^{1/\beta-1}
= -\frac{\beta}{1-\beta}(1-\beta)^{1/\beta}
\geq -\beta,
\]
since $(1-\beta)^{1/\beta}\leq 1-\beta$ for $0<\beta<1.$
We therefore obtain from \eqref{eqn:expPest5} that  
\begin{eqnarray} \label{eqn:expPest6}
\lefteqn{\exp(-e^{-z})-\BP\left(\CT(o) \leq \epsilon u^*+\frac{z}{\mu(\Gamma_o)}\right)^{|K|}}\\
& & \geq \exp(-e^{-z})
-\exp\left(-e^{-z}\left(1 - |A|^{-\epsilon/4}\right) \right) 
\geq -|A|^{-\epsilon/4}. \nonumber
\end{eqnarray}
Combining \eqref{eqn:expPest4} and \eqref{eqn:expPest6} we see that 
\[
\left|\BP\left(\CT(o) \leq \epsilon u^*+\frac{z}{\mu(\Gamma_o)}\right)^{|K|} -\exp(-e^{-z})\right|
\leq 3|A|^{-\epsilon/4}
+|A|^{-\epsilon/4}
=4|A|^{-\epsilon/4}.
\]
Using this and \eqref{eqn:independentest} in 
\eqref{eqn:Tsrhoprel2} proves that 
\[
\left|\BP\left(\CT(A)\leq u^*+\frac{z}{\mu(\Gamma_o)} 
\Big{|} A_\epsilon = K\right)  
- \exp(-e^{-z})\right|
\leq |A|^{-\epsilon}
+4|A|^{-\epsilon/4}
\leq 5|A|^{-\epsilon/4}
\]
proving \eqref{eqn:Tsrhoprel}.
This completes the proof.
\end{proof}

\section{Examples and discussion} \label{sec:examples}

Our main result, Theorem \ref{thm:main}, is geared to work in the 
worst case possible, i.e. when all the points of $A$ are grouped 
close together, and in order to prove Theorem \ref{thm:main} we had
to assume an upper bound on $\kappa_n^{-1},$ i.e. 
that $\kappa_n^{-1}\leq |A_n|^{1-8/(\log \log |A_n|)}.$ As alluded 
to in 
the introduction, we believe that the distribution of the cover time 
may undergo a sort of phase transition as $\kappa_n^{-1}$ increases even
further. In order to indicate this, we will here consider two simple
examples. In the first one we consider the cover time of two widely
separated points, while in the second we consider two almost 
neighboring points. However, we start by observing that 
for $A=\{x\},$ we clearly have from \eqref{eqn:probGreen} that 
\begin{equation}\label{eqn:onepointdistribution}
\BP(\mu(\Gamma_o) \CT(x)\leq u)
=1-\BP(\CT(x)\geq u/\mu(\Gamma_o))
=1-\left(G^{o,o}\right)^{-\frac{u}{\mu(\Gamma_o)}}
=1-e^{-u},
\end{equation} 
where we used \eqref{eqn:mulogG} in the last equality. Therefore, 
$\mu(\Gamma_o) \CT(x)$ is always an exponentially distributed random 
variable with parameter one.

\begin{example} \label{ex:twopointwidesep}
\end{example}
\noindent
Consider a set with two points, say $A=\{o,x\}$ and assume for 
convenience that $|x|$ is even. Then, we start 
by noticing that 
\begin{eqnarray} \label{eqn:Ex1prel}
\lefteqn{\BP(\CT(o,x)\leq u)=\BP(\CT(o)\leq u,\CT(x)\leq u)}\\
& &=2\BP(\CT(o)\leq u)-\BP(\CT(o)\leq u \cup \CT(x)\leq u) 
\nonumber\\
& &  =2-2\BP(\CT(o)\geq u)
-\left(1-\BP(\CT(o)\geq u, \CT(x)\geq u)\right) \nonumber\\
& & =1-2\BP(\CT(o)\geq u)+\BP(\CT(o)\geq u, \CT(x)\geq u) \nonumber\\
& & =1-2\BP(x \cap \CC_u=\emptyset)
+\BP(\{o,x\} \cap \CC_u=\emptyset). \nonumber
\end{eqnarray}
Using this, and \eqref{eqn:probxyuncov} we then see that 
\begin{eqnarray} \label{eqn:Ex1prel1}
\lefteqn{\BP(\CT(o,x)\leq u)-\BP(\CT(o)\leq u)^2}\\
& & =1-2\BP(x \cap \CC_u= \emptyset)
+\BP(\{o,x\} \cap \CC_u= \emptyset)
-(1-\BP(x \cap \CC_u= \emptyset))^2 \nonumber \\
& & =\BP(\{o,x\} \cap \CC_u= \emptyset)
-\BP(x \cap \CC_u= \emptyset)^2\nonumber \\
& & =\BP(x \cap \CC_u= \emptyset)^2
\left(\left(1-\left(\frac{G^{o,x}}{G^{o,o}}\right)^2\right)^{-u}-1\right). \nonumber
\end{eqnarray}
Next, we trivially have that 
\[
\left(1-\left(\frac{G^{o,x}}{G^{o,o}}\right)^2\right)^{-u}
\geq 1,
\]
and so it follows from \eqref{eqn:Ex1prel1} that 
\begin{equation} \label{eqn:Ex1prel2}
\BP(\CT(o,x)\leq u)-\BP(\CT(o)\leq u)^2 \geq 0.
\end{equation}
In order to bound the expression in \eqref{eqn:Ex1prel1} from above, 
we assume now that $|x|\geq 10 \kappa^{-2}$.
We then have that (since $|x|$ is assumed to be even)
\begin{eqnarray*}
\lefteqn{G^{o,x}
=\sum_{n=|x|}^{\infty}\left(\frac{1}{4+\kappa}\right)^{n} W_n^{o,x}
=\sum_{n=|x|/2}^{\infty}\left(\frac{1}{4+\kappa}\right)^{2n} 
W_{2n}^{o,x}}\\
& & \leq \sum_{n=|x|/2}^{\infty}\left(\frac{1}{4+\kappa}\right)^{2n} 
W_{2n}^{o,o}
\leq 4e^{-(|x|/2-1)\kappa/4}
\leq 4e^{-(10\kappa^{-2}/2-1)\kappa/4}
\leq 4e^{-\kappa^{-1}}.
\end{eqnarray*}
where we used Lemma \ref{lemma:WxWo} in the first inequality and 
\eqref{eqn:Wootailest_alt} in the second. Furthermore, 
as in the proof of \eqref{eqn:sumest2}, we see that 
\begin{eqnarray} \label{eqn:Ex1second}
\lefteqn{
\left(1-\left(\frac{G^{o,x}}{G^{o,o}}\right)^2\right)^{-u}}\\
& & \leq \left(1-\left(\frac{4 e^{-\kappa^{-1}}}{G^{o,o}}\right)^2\right)^{-u}
=\exp\left(-u\log \left(1-\left(\frac{4 e^{-\kappa^{-1}}}{G^{o,o}}\right)^2\right)\right) 
\nonumber\\
& &\leq \exp\left(2u\left(\frac{4 e^{-\kappa^{-1}}}{G^{o,o}}\right)^2\right)
\leq \exp\left(32ue^{-2\kappa^{-1}}\right)
\leq 1+32ue^{-2\kappa^{-1}} \nonumber
\end{eqnarray}
for $\kappa^{-1}$ large enough and since $G^{o,o}\geq 1.$
Combining \eqref{eqn:Ex1prel1} with \eqref{eqn:Ex1second} we conclude that 
\begin{equation} \label{eqn:Ex1prel3}
\BP(\CT(o,x)\leq u)-\BP(\CT(o)\leq u)^2
\leq \BP(x \cap \CC_u= \emptyset)^2 32ue^{-2\kappa^{-1}}.
\end{equation}
By fixing $u$ and letting $\kappa^{-1}\to \infty,$ we therefore 
see that (by using \eqref{eqn:muGammaolowest} and 
\eqref{eqn:probGreenbasic})
\begin{eqnarray} \label{eqn:Ex1final}
\lefteqn{|\BP(\mu(\Gamma_o)\CT(o,x)\leq u)
-\BP(\mu(\Gamma_o)\CT(o)\leq u)^2|}\\
& & \leq \BP(x \cap \CC_{\frac{u}{\mu(\Gamma_o)}}= \emptyset)^2
32\frac{u}{\mu(\Gamma_o)}e^{-2\kappa^{-1}}
=32ue^{-2u}\frac{e^{-2\kappa^{-1}}}{\mu(\Gamma_o)} 
 \nonumber\\
& & \leq 
32ue^{-2u}\frac{e^{-2\kappa^{-1}}}{\log \log \kappa^{-1}-\log \pi}
\to 0, \nonumber
\end{eqnarray}
as $\kappa^{-1}\to \infty.$ Thus, the re-scaled cover time 
$\CT(o,x)$ behaves asymptotically like the cover time of two independent 
exponentially distributed random variables with parameter one.
In light of the great distance between
$o$ and $x,$ this is not surprising.

\medskip

In our next example, we will consider two points which are close. 
We will need the following lemma, whose proof is similar to that of Lemma \ref{lemma:partialsumlowest} and is deferred to the appendix.

\begin{lemma} \label{lemma:Go11lowest}
We have that for any $\kappa>0,$ 
\begin{equation}
G^{o,(1,1)}
\geq \frac{\log \kappa^{-1}}{\pi}-1.
\end{equation}
\end{lemma}

\begin{example} \label{ex:twopointneigh}
\end{example}
\noindent This example is similar to Example 2 in that 
$A=\{o,x\}$. However, we will here assume that $x=(1,1)$ so that 
the two points are close to each other. We start by noting that  
by \eqref{eqn:Gooupest} and Lemma
\ref{lemma:Go11lowest} 
we have that 
\[
G^{o,o}-G^{o,x}
\leq \frac{\log \kappa^{-1}}{\pi}+2
-\left(\frac{\log \kappa^{-1}}{\pi}-1\right)=3.
\]
We therefore have that 
\begin{eqnarray*}
\lefteqn{\left(1-\left(\frac{G^{o,x}}{G^{o,o}}\right)^2\right)^{-u}
\geq 
\left(1-\left(\frac{G^{o,o}-3}{G^{o,o}}\right)^2\right)^{-u}}\\
& & =\left(1-\left(1-\frac{3}{G^{o,o}}\right)^2\right)^{-u}
= \left(\frac{6}{G^{o,o}}-\frac{9}{(G^{o,o})^2}\right)^{-u}\\
& & \geq \left(\frac{6}{G^{o,o}}\right)^{-u}
=\frac{6^{-u}}{\BP(x \cap \CC_{u}= \emptyset)},
\end{eqnarray*}
and therefore (as in \eqref{eqn:Ex1second})
\begin{eqnarray} \label{eqn:Ex2first}
\lefteqn{\BP(\{o,x\}\cap \CC_u =\emptyset)
=\BP(x \cap \CC_u\neq \emptyset)^2
\left(1-\left(\frac{G^{o,x}}{G^{o,o}}\right)^2\right)^{-u}}\\
& & \geq \BP(x \cap \CC_u= \emptyset)^2
\frac{6^{-u}}{\BP(x \cap \CC_{u}=\emptyset)}
=\BP(x \cap \CC_u= \emptyset)6^{-u}. \nonumber
\end{eqnarray}
Noting that trivially, 
\[
\BP(\CT(o)\leq u/\mu(\Gamma_o))
-\BP(\CT(o,x)\leq u/\mu(\Gamma_o))\geq 0,
\]
we can therefore see that 
by using \eqref{eqn:Ex2first} and \eqref{eqn:Ex1prel}
\begin{eqnarray*}
\lefteqn{|\BP(\mu(\Gamma_o)\CT(o,x)\leq u)
-\BP(\mu(\Gamma_o)\CT(o)\leq u)|}\\
& & =\BP(\CT(o)\leq u/\mu(\Gamma_o))
-\BP(\CT(o,x)\leq u/\mu(\Gamma_o)) \nonumber \\
& & =1-\BP(x \cap \CC_{u/\mu(\Gamma_o)}= \emptyset)
-(1-2\BP(x \cap \CC_{u/\mu(\Gamma_o)}= \emptyset)
+\BP(\{o,x\} \cap \CC_{u/\mu(\Gamma_o)}= \emptyset))\\
& & =\BP(x \cap \CC_{u/\mu(\Gamma_o)}= \emptyset)
-\BP(\{o,x\} \cap \CC_{u/\mu(\Gamma_o)}= \emptyset)
\leq \BP(x \cap \CC_{u/\mu(\Gamma_o)} = \emptyset)
\left(1-6^{-{u/\mu(\Gamma_o)}} \right)\\
& & =e^{-u}\left(1-6^{-{u/\mu(\Gamma_o)}} \right)
\to 0, 
\end{eqnarray*}
when $\kappa^{-1} \to \infty,$ since it follows from \eqref{eqn:Goolowest}
that in this case $\mu(\Gamma_o) \to \infty.$ 
We conclude that the re-scaled cover time of 
$\CT(o,x)$ behaves like a single exponentially 
distributed random variable with parameter one.

\medskip

Recall the discussion in the Introduction concerning a possible 
phase transition depending on the rate at which $\kappa_n \to 0$. 
The purpose of our next example is to demonstrate 
the (perhaps unsurprising) fact that if we allow the separation distance 
between the vertices in $A_n$ to depend on the killing rate $\kappa_n,$
such a phase transition will be absent. 
Allowing the separation distance to depend on $\kappa_n$ may feel like
``cheating'', but serves to demonstrate how the 
geometry of the sets $A_n$ plays an important role. 
In order not to make this example, and
indeed the entire paper, forbiddingly long, we shall be somewhat 
informal. Otherwise we would have to repeat large parts of Sections
\ref{sec:2ndmom} and \ref{sec:proofofmain}.

\begin{example} \label{ex:manypointswidesep}

\end{example}
Consider a (large) set $A$ and a killing rate $\kappa$ such that
$\kappa^{-1}\geq \log |A|,$ 
\begin{equation} \label{eqn:xylargesep}
|x-y|\geq 10\kappa^{-2}
\end{equation}
for every $x,y\in A,$ and such that $|x-y|$ is even for every 
$x,y\in A.$ We then have that
\begin{eqnarray*}
\lefteqn{G^{o,x-y}
=\sum_{n=|x|}^{\infty}\left(\frac{1}{4+\kappa}\right)^{n} W_n^{o,x-y}
=\sum_{n=|x|/2}^{\infty}\left(\frac{1}{4+\kappa}\right)^{2n} 
W_{2n}^{o,x-y}}\\
& & \leq \sum_{n=|x|/2}^{\infty}\left(\frac{1}{4+\kappa}\right)^{2n} 
W_{2n}^{o,o}
\leq 4e^{-(|x|/2-1)\kappa/4}
\leq 4e^{-(10\kappa^{-2}/2-1)\kappa/4}
\leq 4e^{-\kappa^{-1}},
\end{eqnarray*}
where we used Lemma \ref{lemma:WxWo} in the first inequality and 
\eqref{eqn:Wootailest_alt} in the second.
Then, as in Lemma \ref{lemma:doublesumlargedist} we have that 
(using that $\mu(\Gamma_o)\geq 1$ whenever $\kappa^{-1}$ is 
larger than $e^9$ due to \eqref{eqn:muGammaolowest}),
\begin{eqnarray*}
\lefteqn{\BP(x,y\in A_\epsilon)
\leq |A|^{-2(1-\epsilon)}
\exp\left(2(1-\epsilon)\frac{\log |A|}{\mu(\Gamma_o)}
\left(\frac{G^{o,y-x}}{G^{o,o}}\right)^2\right)} \\
& & \leq |A|^{-2(1-\epsilon)}
\exp\left(2\log |A| \left(16 e^{-2\kappa^{-1}}\right)\right) \\
& & \leq |A|^{-2(1-\epsilon)}
\exp\left(e^{-\kappa^{-1}}\right)
\leq |A|^{-2(1-\epsilon)}\left(1+e^{-\kappa^{-1}}\right),
\end{eqnarray*}
where we used the assumption that $\kappa^{-1}\geq \log |A|$
in the penultimate inequality.

Next, we consider sequences $(A_n,\kappa_n)_{n \geq 1}$
of sets and killing times 
with the property of $(A, \kappa)$
above, and such that $|A_n| \to \infty.$ We can then use 
the machinery of Sections \ref{sec:2ndmom}
and \ref{sec:proofofmain} and we note that we only need to 
consider the case where $x,y$ is ``well separated'' (i.e 
Lemma \ref{lemma:doublesumlargedist}). Applying this machinery 
demonstrates that that no upper bound for 
$\kappa_n^{-1}$ is needed when proving a statement analogous to 
Theorem \ref{thm:main} in this case. This demonstrates that if we always have
a large separation between the points in $A_n,$ such as described 
by \eqref{eqn:xylargesep}, we will obtain a Gumbel distribution as the limit
even when $\kappa_n \to 0$ exceedingly fast.

\medskip

We end this section with an informal discussion (this is the 
discussion mentioned in the Introduction)
concerning the case
where $\kappa_n^{-1}>>|A_n|$ and $A_n$ is a (very large) 
ball or square. In this case, we 
believe that it may be that for sufficiently large values of 
$\kappa_n^{-1},$ 
the set $A_n$ will asymptotically be covered by the first loop 
which touches the
set $A_n.$ The reason for this belief can be explained in two 
steps as follows.

{\bf Step 1:}
After an exponentially distributed time with rate $\mu(\Gamma_{A_n})$
where
\[
\Gamma_{A_n}:=\bigcup_{x \in A_n} \Gamma_x,
\]
the first loop that touches $A_n$ appears. With very high probability 
this loop should be of length order at least $\log \kappa_n^{-1}$. 
The reason why we expect this, is that when analyzing $\mu(\Gamma_o)$
starting from \eqref{eqn:muGammao}, one can show that the 
contribution from 
loops of length $n$ smaller than $\log \kappa_n^{-1}$ will be 
small compared to the total sum, which is of order 
$\log \log \kappa_n^{-1}$ (due to \eqref{eqn:muGammaoest} and 
\eqref{eqn:muGammaolowest}).

{\bf Step 2:} Considering a typical loop of size order at least
$\log \kappa_n^{-1}$ from Step 1 touching $A_n$, then the probability 
that it will in fact cover the entirety of $A_n$ is again very high
(note that therefore, 
$\mu(\Gamma_{A_n}) \approx \mu(\Gamma_o)$) whenever $\kappa_n^{-1}$
is large enough. The reason why we 
believe this to be true, stems from considering the 
corresponding problem for 
a simple symmetric random walk $S_n$ started at the origin of
$\BZ^2.$ Let $T_n$ be the first time when the walk has visited every 
site $x\in B_n,$ where $B_n$ is the ball of radius $n$ in $\BZ^2.$ 
According to \cite{DPRZ} (see also the references within for
background on this challenging problem), we have that 
\begin{equation} \label{eqn:srwcovertime}
\lim_{n \to \infty} \BP(\log T_n \leq t(\log n)^2)=e^{-4/t}.
\end{equation}
From this, one can then conclude that ``with high probability'', 
the ball $B_n$ will be covered at time, say, exponential of $\gamma(n)(\log n)^2$
where $\gamma(n)$ is chosen appropriately. By letting 
\[
Tr(n):=\{x\in \BZ^2: S_k=x \textrm{ for some } k=0,1,\ldots,n\}
\]
denote the trace of the random walk until time $n$ and observing that
\[
\BP(\log T_n \leq t(\log n)^2) = \BP(T_n \leq n^{t(\log n)}) = \BP(B_n \subset Tr(n^{t \log n})),
\]
it follows from \eqref{eqn:srwcovertime} that 
\begin{equation} \label{eqn:srwcovertime_alt}
\lim_{n \to \infty} \BP(B_n \subset Tr(n^{t \log n}))=e^{-4/t}.
\end{equation}
In turn, one can hope 
that a similar statement can be inferred for a loop rooted at the 
origin by simply conditioning the random walk to be back at the origin 
at some suitable time. For instance, from a statement along the 
lines of (ignoring that $n^{t \log n}$ may not be an integer) 
\begin{equation} \label{eqn:loopcovertime}
\lim_{n \to \infty} \BP(B_n \subset Tr(n^{t \log n})
|S_{n^{t \log n}}=o)=e^{-4/t},
\end{equation}
one can infer that with very high probability, the ball $B_n$ will 
be covered by a loop of length $n^{\gamma(n) \log n}$ where again 
$\gamma(n)$ is chosen appropriately. However, there does not seem
to be an easy way to infer \eqref{eqn:loopcovertime} directly
from \eqref{eqn:srwcovertime_alt} without knowing an explicit
rate of convergence in \eqref{eqn:srwcovertime_alt}. The issue is 
of course that we are conditioning on an event which is known 
to have probability of order $\left(n^{t \log n}\right)^{-1}.$
In order to turn the intuition above into a proof, one would instead
have to prove \eqref{eqn:loopcovertime} by other means. One may
attempt to adapt the techniques of \cite{DPRZ}, but even if this
were possible, that may very well be an entire project in 
itself and outside the scope of this paper.

\begin{appendices} 
\section{}\label{app:walks}

In this appendix we shall provide full proofs of all lemmas of Section 
\ref{sec:Green}. \\

\noindent
\begin{proof}[Proof of Lemma \ref{lemma:WxWo}.]
Let $S_n^x$ denote a simple symmetric random walk started at $x$ 
and with killing rate $\kappa=0.$ Since 
$\BP(S_{2n}^o=x)=4^{-2n} W_{2n}^{o,x},$ it suffices to show that 
$\BP(S_{2n}^o=x)\leq\BP(S_{2n}^o=o)$ for every $x$ such that $|x|$
is even. To that end, we observe that 
\begin{eqnarray*}
\lefteqn{\BP(S_{2n}^o=x)
=\sum_y \BP(S_{n}^o=y)\BP(S_{n}^y=x)
=\sum_y \BP(S_{n}^o=y)\BP(S_{n}^x=y)}\\
& & \leq \sqrt{\sum_y \BP(S_{n}^o=y)\BP(S_{n}^o=y)}
\sqrt{\sum_y \BP(S_{n}^x=y)\BP(S_{n}^x=y)} \\
& & = \sqrt{\sum_y \BP(S_{n}^o=y)\BP(S_{n}^y=o)}
\sqrt{\sum_y \BP(S_{n}^x=y)\BP(S_{n}^y=x)}\\
& &  = \sqrt{\BP(S_{2n}^o=o)}
\sqrt{\BP(S_{2n}^x=x)}
=\BP(S_{2n}^o=o),
\end{eqnarray*}
where we used the reversibility of the random walk in the second 
and third equality.
\end{proof}

For the proof of Lemma \ref{lemma:partialsumlowest} 
(and for future reference), we state
a standard Stirling estimate (see \cite{Robbins}):
\begin{equation} \label{eqn:Stirling}
\sqrt{2 \pi}n^{n+1/2} e^{-n}\leq n! 
\leq \sqrt{2 \pi }n^{n+1/2} e^{-n}e^{1/(12n)} \; \text{ for every } n \geq 1.
\end{equation}

\medskip 

\noindent
\begin{proof}[Proof of Lemma \ref{lemma:partialsumlowest}.]
Recall that $L_{2n}$ denotes the number of loops rooted at $o$ of length 
$2n,$ and note that $W^{o,o}_{2n}=L_{2n}.$ It is well known 
(see for instance \cite{DE}) that $L_{2n}={2n \choose n}^2$ and 
furthermore, by \eqref{eqn:Stirling}, we have that
\begin{eqnarray*}
\lefteqn{{2n \choose n}=\frac{(2n)!}{n!n!}
\geq \frac{\sqrt{2 \pi}(2n)^{2n+1/2} e^{-2n}}
{2 \pi  n^{2n+1} e^{-2n}e^{1/(6n)}}}\\
& & =\frac{1}{\sqrt{2 \pi}}\frac{2^{2n+1/2} n^{2n+1/2}}{n^{2n+1}}e^{-1/(6n)}
=\frac{1}{\sqrt{\pi}}\frac{4^n }{\sqrt{n}}e^{-1/(6n)},
\end{eqnarray*}
and so 
\[
L_{2n}={2n \choose n}^2\geq \frac{4^{2n}}{\pi n}e^{-1/(3n)}.
\]
We can therefore conclude that 
\begin{equation} \label{eqn:sumest}
\sum_{n=0}^{N-1} \left(\frac{1}{4+\kappa}\right)^{2n}W^{o,o}_{2n}
=1+\sum_{n=1}^{N-1} \left(\frac{1}{4+\kappa}\right)^{2n}L_{2n}
\geq1+
\sum_{n=1}^{N-1} \left(\frac{4}{4+\kappa}\right)^{2n} 
\frac{e^{-1/(3n)}}{ \pi n}.
\end{equation}
Next, let $f(x)=\frac{e^{-1/(3x)}}{x}$ and observe that 
\[
f'(x)=\frac{1}{x}\frac{e^{-1/(3x)}}{3x^2}-\frac{1}{x^2}e^{-1/(3x)}
= \frac{e^{-1/(3x)}}{x^2} \left( \frac{1}{3x}-1 \right),
\]
and so $f(x)$ is decreasing for all $x\geq 1.$ Therefore,
\begin{eqnarray} \label{eqn:partialsumlowintest}
\lefteqn{\sum_{n=1}^{N-1} \frac{e^{-1/(3n)}}{\pi n}
\left(\frac{4}{4+\kappa}\right)^{2n}
\geq \sum_{n=1}^{N-1} \int_{n}^{n+1} \frac{e^{-1/(3x)}}{ \pi x}
\left(\frac{1}{1+\kappa/4}\right)^{2x} dx }\\
& & = \int_1^{N} \frac{e^{-1/(3x)}}{ \pi x}
\left(\frac{1}{1+\kappa/4}\right)^{2x} dx 
=\int_{\kappa}^{N \kappa} \frac{e^{-\kappa/(3y)}}{\pi y}
\left(\frac{1}{1+\kappa/4}\right)^{2y/\kappa} dy. \nonumber
\end{eqnarray}
Then, observe that for any $x>0$ we have that $\log (1+x)\leq 2x$.
Therefore, by also using that $e^{-x}\geq 1-x$ for every $x>0,$ we see that
\begin{eqnarray*}
\lefteqn{\left(\frac{1}{1+\kappa/4}\right)^{2y/\kappa}
=\exp\left(-\frac{2y}{\kappa}\log(1+\kappa/4)\right)}\\
& & \geq 1-\frac{2y}{\kappa}\log(1+\kappa/4)
\geq 1-\frac{2y}{\kappa}\frac{\kappa}{2}
=1-y,
\end{eqnarray*}
and so, by again using that $e^{-x}\geq 1-x$ for every $x>0,$
\begin{eqnarray}\label{eqn:intest}
\lefteqn{\int_{\kappa}^{N\kappa} \frac{e^{-\kappa/(3y)}}{\pi y}
\left(\frac{1}{1+\kappa/4}\right)^{2y/\kappa} dy 
\geq \int_{\kappa}^{N \kappa} \frac{1}{\pi y}
\left(1-y\right)\left(1-\frac{\kappa}{3y}\right)dy }\\
& & =\int_{\kappa}^{N \kappa} \frac{3+\kappa}{3 \pi y}
-\frac{1}{\pi}-\frac{\kappa}{3 \pi y^2}dy 
\geq \int_{\kappa}^{N \kappa}  \frac{1}{\pi y}
-\frac{1}{\pi}-\frac{\kappa}{3 \pi y^2}  dy \nonumber \\
& &  =\frac{\log N}{\pi}-\frac{(N-1)\kappa}{\pi}
+\left[\frac{\kappa}{3 \pi y}\right]_\kappa^{N\kappa}
\geq \frac{\log N}{\pi}-\frac{N\kappa}{\pi}
+\frac{\kappa}{3\pi}\left(\frac{1}{N\kappa}-\frac{1}{\kappa}\right) 
\nonumber \\
& &
\geq \frac{\log N}{\pi}-\frac{N\kappa}{\pi}-\frac{1}{3 \pi }. \nonumber
\end{eqnarray}
Combining \eqref{eqn:sumest}, \eqref{eqn:partialsumlowintest} 
and \eqref{eqn:intest} gives us that 
\[
\sum_{n=0}^{N-1} \left(\frac{1}{4+\kappa}\right)^{2n}W^{o,o}_{2n}
\geq 1+\frac{\log N}{\pi}-\frac{N\kappa}{\pi}-\frac{1}{3 \pi },
\]
establishing \eqref{eqn:partialsumlowest1}.

In order to obtain \eqref{eqn:Goolowest}, we first observe that 
by a minor modification of \eqref{eqn:intest} (changing the second to last
inequality into an equality),
we have that 
\[
\sum_{n=0}^{N-1} \left(\frac{1}{4+\kappa}\right)^{2n}W^{o,o}_{2n}
\geq 1+\frac{\log N}{\pi}-\frac{(N-1)\kappa}{\pi}-\frac{1}{3 \pi }.
\]
Thus,
\begin{eqnarray*} 
\lefteqn{G^{o,o}
=\sum_{n=0}^{\infty} \left(\frac{1}{4+\kappa}\right)^{2n}W^{o,o}_{2n} 
\geq \sum_{n=0}^{\lfloor \kappa^{-1}\rfloor} 
\left(\frac{1}{4+\kappa}\right)^{2n}W^{o,o}_{2n} }\\
& & \geq 1+
\frac{\log (\lfloor\kappa^{-1}\rfloor+1)}{\pi}
-\frac{\lfloor \kappa^{-1}\rfloor \kappa}{\pi}-\frac{1}{3 \pi }
\geq 1+\frac{\log \kappa^{-1}}{\pi}-\frac{1}{\pi}-\frac{1}{3 \pi},
\end{eqnarray*}
as desired.
\end{proof}

We now turn to Lemma \ref{lemma:partialsumupest}.\\

\noindent
\begin{proof}[Proof of Lemma \ref{lemma:partialsumupest}.]
We will again use the Stirling estimate \eqref{eqn:Stirling}
to see that
\begin{eqnarray*}
\lefteqn{{2n \choose n}=\frac{(2n)!}{n!n!}
\leq \frac{\sqrt{2 \pi }(2n)^{2n+1/2} e^{-2n}e^{1/(24n)}}
{2 \pi  n^{2n+1} e^{-2n}}}\\
& & =\frac{1}{\sqrt{2 \pi}}\frac{2^{2n+1/2} n^{2n+1/2}}{n^{2n+1}}e^{1/(24n)}
=\frac{1}{\sqrt{\pi}}\frac{4^n }{\sqrt{n}}e^{1/(24n)},
\end{eqnarray*}
and so 
\[
W_{2n}^{o,o}=L_{2n}={2n \choose n}^2\leq \frac{4^{2n}}{\pi n}e^{1/(12n)}.
\]
We obtain that for $N\geq 1,$ 
\[
\sum_{n=N+1}^\infty \left(\frac{1}{4+\kappa}\right)^{2n} W_{2n}^{o,o}
\leq \sum_{n=N+1}^\infty \left(\frac{1}{4+\kappa}\right)^{2n} 
\frac{4^{2n}}{\pi n}e^{1/(12n)}.
\]
Clearly $\left(\frac{4}{4+\kappa}\right)^{2x} \frac{e^{1/(12x)}}{x}$ is a
decreasing function of $x>0,$ and so we get that 
\begin{eqnarray*}
\lefteqn{\sum_{n=N+1}^{\infty} \left(\frac{4}{4+\kappa}\right)^{2n} 
\frac{e^{1/(12n)}}{ \pi n}
\leq \int_N^\infty\left(\frac{4}{4+\kappa}\right)^{2x} 
\frac{e^{1/(12x)}}{ \pi x} dx} \\
& & \leq \int_N^\infty\left(\frac{4}{4+\kappa}\right)^{2x} 
\left(\frac{1}{ \pi x} +\frac{1}{ 6\pi x^2}\right)dx
=\int_{N\kappa}^\infty\left(\frac{4}{4+\kappa}\right)^{2y/\kappa} 
\left(\frac{1}{ \pi y} +\frac{\kappa}{ 6\pi y^2}\right)dy,
\end{eqnarray*}
where we used that $e^y\leq 1+2y$ for $0<y<1$ in the second inequality.
Furthermore, we have that $\log(1+z)\geq z/2$ for every $0<z<1,$ 
and therefore 
\[
\left(\frac{1}{1+\kappa/4}\right)^{2y/\kappa} 
=\exp\left(-\frac{2y}{\kappa}\log(1+\kappa/4)\right)
\leq \exp\left(-\frac{2y}{\kappa}\frac{\kappa}{8}\right)
=e^{-y/4},
\]
and so we obtain 
\begin{equation} \label{eqn:app1}
\sum_{n=N+1}^\infty \left(\frac{1}{4+\kappa}\right)^{2n} 
W_{2n}^{o,o}
\leq \int_{N\kappa}^\infty e^{-y/4}
\left(\frac{1}{ \pi y} +\frac{\kappa}{ 6\pi y^2}\right)dy.
\end{equation}
Therefore, if $N<\kappa^{-1}$ we see that
\begin{eqnarray*}
\lefteqn{\sum_{n=N+1}^\infty \left(\frac{1}{4+\kappa}\right)^{2n} 
W_{2n}^{o,o}
\leq \int_{N\kappa}^\infty e^{-y/4}
\left(\frac{1}{ \pi y} +\frac{\kappa}{ 6\pi y^2}\right)dy}\\
& & =\int_{N\kappa}^1 e^{-y/4}
\left(\frac{1}{ \pi y} +\frac{\kappa}{ 6\pi y^2}\right)dy
+\int_{1}^\infty e^{-y/4}
\left(\frac{1}{ \pi y} +\frac{\kappa}{ 6\pi y^2}\right)dy\\
& & \leq \int_{N\kappa}^1 
\left(\frac{1}{ \pi y} +\frac{\kappa}{ 6\pi y^2}\right)dy
+\int_{1}^\infty e^{-y/4}dy \\
& & =\frac{\log (N\kappa)^{-1}}{\pi}
+\frac{\kappa}{6 \pi}\left[\frac{-1}{y}\right]_{N\kappa}^1
+\int_{1}^\infty e^{-y/4}dy 
\leq \frac{\log (N\kappa)^{-1}}{\pi}
+\frac{1}{6 \pi N}+4,
\end{eqnarray*}
proving \eqref{eqn:Wootailest}. 
If $N\kappa\geq 1/2$, we see from \eqref{eqn:app1} that 
\[
\sum_{n=N+1}^\infty \left(\frac{1}{4+\kappa}\right)^{2n} 
W_{2n}^{o,o}
\leq \int_{N\kappa}^\infty e^{-y/4}
\left(\frac{2}{ \pi} +\frac{4\kappa}{ 6\pi}\right)dy
\leq 
\int_{N\kappa}^\infty e^{-y/4}dy =4e^{-N\kappa/4},
\]
proving \eqref{eqn:Wootailest_alt}. 

For \eqref{eqn:Gooupest}, observe that as above, 
\begin{eqnarray}\label{eqn:Gooupest1}
\lefteqn{G^{o,o}=1+\left(\frac{1}{4+\kappa}\right)^2 W_2^{o,o}
+\sum_{n=2}^{\infty} \left(\frac{1}{4+\kappa}\right)^{2n} W_{2n}^{o,o}}\\
& & \leq 1+\frac{1}{4}
+\int_{\kappa}^\infty e^{-y/4}
\left(\frac{1}{ \pi y} +\frac{\kappa}{ 6\pi y^2}\right)dy. \nonumber
\end{eqnarray} 
Next, we have that 
\begin{eqnarray*}
\lefteqn{\int_{\kappa}^1 e^{-y/4}
\left(\frac{1}{ \pi y} +\frac{\kappa}{ 6\pi y^2}\right)dy 
\leq \int_{\kappa}^1
\left(\frac{1}{ \pi y} +\frac{\kappa}{ 6\pi y^2}\right)dy }\\
& & =\left[\frac{\log y}{\pi}-\frac{\kappa}{6 \pi y}\right]_{\kappa}^1
=-\frac{\log \kappa}{\pi}-\frac{\kappa}{6 \pi} +\frac{1}{6 \pi}
\leq \frac{\log \kappa^{-1}}{\pi}+\frac{1}{6 \pi },
\end{eqnarray*}
and furthermore, using that $\kappa<1,$
\[
\int_1^\infty e^{-y/4}
\left(\frac{1}{ \pi y} +\frac{\kappa}{ 6\pi y^2}\right)dy 
\leq \frac{7}{6\pi}\int_1^\infty e^{-y/4}\frac{1}{y}dy 
\leq \frac{2}{\pi},
\]
where the last inequality can be verified through elementary but 
lengthy calculations.
We can then conclude from \eqref{eqn:Gooupest1} that 
\[
G^{o,o}\leq 1+\frac{1}{4}+\frac{\log \kappa^{-1}}{\pi}+\frac{1}{6 \pi }
+\frac{2}{\pi}\leq \frac{\log \kappa^{-1}}{\pi}+2,
\]
as desired.
\end{proof}

Before we can prove Lemma \ref{lemma:sumWoxest}, which is 
the last lemma of Section \ref{sec:Green}, we need to present a version
of the so-called local central limit theorem for random walks. 
We do not claim that the following result is original, although
we could not find an explicit reference. However, as we shall see,
the result is an 
easy consequence of Theorem 2.1.1 equation (2.4) of \cite{LL}.

\begin{lemma} \label{lemma:2dLCLT}
There exists a constant $C<\infty,$ such that for any $n\geq 1$ and any 
$x\in \BZ^2$ such that $|x|\geq 3,$ we have that 
\begin{equation} \label{eqn:2dLCLT}
\BP(S_n=x)\leq \frac{2}{n}e^{-\frac{|x|^2}{2n}}+\frac{C}{n|x|^2}.
\end{equation} 
\end{lemma}
\begin{proof}
Let $(S_n^a)_{n \geq 1}$ be an aperiodic symmetric random walk 
on $\BZ^2.$ It follows from equation (2.4) of \cite{LL} that 
for any $n\geq 1$ and $y\in \BZ^2,$
\begin{equation} \label{eqn:2dLCLTaper}
\BP(S^a_n=y)
\leq \frac{1}{2 \pi n \sqrt{\det \Gamma}}e^{-\frac{\CJ^*(y)^2}{2n}}
+\frac{C}{n\Vert y \Vert^2},
\end{equation}
where $\Vert \cdot \Vert$ denotes the $L^2$-norm. Furthermore, 
using the notation in \cite{LL} (see in particular p.\ 4), 
$\Gamma$ is the covariance matrix of the walk
and $\CJ^*(y)^2=\Vert y \cdot \Gamma^{-1} y \Vert$.

For a simple symmetric random walk on $\BZ^2,$ it is easy to see that 
$\Gamma=\frac{1}{2}I$ where $I$ is the $2\times 2$ identity matrix, and that 
$\CJ^*(y)^2=2\Vert y\Vert^2$ as stated on p.\ 4 of 
\cite{LL}. It is worth noting that in \cite{LL} 
the notation $|\cdot|$ is used for the $L^2$-norm (corresponding to Euclidean distance) while, throughout this paper, we use this notation for the $L^1$-norm (corresponding to graph distance), as defined at the beginning of Section~\ref{sec:Green}.
The simple symmetric random walk 
is however not aperiodic but rather bipartite (in the sense of p.\ 3 of 
\cite{LL}) since $\BP(S_n=o)=0$ whenever $n$ is odd. Therefore, we 
cannot apply \eqref{eqn:2dLCLTaper} directly.
Instead, we shall first consider even times $n,$ and turn the simple 
symmetric walk at those times into an aperiodic 
random walk on $\BZ^2$ as we now explain. This auxiliary walk
$(S^e_n)_{n\geq 1}$ is defined by letting 
\begin{equation} \label{eqn:defSne}
S^e_n=T(S_{2n}) \textrm{ where } 
T(x)=T((x_1,x_2))=\frac{1}{2}(x_1+x_2,x_2-x_1),
\end{equation}
for every $n\geq 0$ and $x\in \BZ^2.$
This corresponds to considering the walk at even times on the even sublattice of $\BZ^2$ rotated clockwise by 45 degrees and shrunk by a factor of $\sqrt{2}$, which gives a new walk on $\BZ^2.$
For example, $S_2$ can take any of the nine values in 
$\{o,\pm 2e_1,\pm 2e_2,\pm e_1\pm e_2\}$ 
(where $e_1=(1,0)$
and $e_2=(0,1)$), which correspond to the vertices at distance 0 
or distance 2 from $o.$ The mapping $T$ then maps these 
nine vertices to the set $\{o,\pm e_1,\pm e_2,\pm e_1 \pm e_2\}$
so that for instance 
\[
T(2 e_1)=T((2,0))=\frac{1}{2}(2,-2)=(1,-1)
\]
and 
\[
T(-e_1-e_2)=T((-1,-1))=\frac{1}{2}(-2,0)=(-1,0).
\]
Since the steps of the random walk $(S_n)_{n\geq 1}$ are independent, 
we see that $(S^e_n)_{n\geq 1}$ is an aperiodic random walk on $\BZ^2$
(since $\BP(S_n=o)>0$ for every $n\geq 0$).
Therefore, we can apply \eqref{eqn:2dLCLTaper} to this walk.

Fix $x\in \BZ^2$ such that $|x|$ is even. Clearly, $\BP(S_n=x)=0$
if $n$ is odd and so trivially \eqref{eqn:2dLCLT} holds in this 
case. If instead $n$ is even, we have from \eqref{eqn:2dLCLTaper} that
\begin{equation} \label{eqn:Psnxeven}
\BP(S_n=x)=\BP(S^e_{n/2}=T(x))
\leq \frac{1}{2 \pi n \sqrt{\det \Gamma}}e^{-\frac{\CJ^*(T(x))^2}{2n}}
+\frac{C'}{n\Vert T(x) \Vert^2},
\end{equation}
for some $C'<\infty$.
In order to continue, we need to determine 
$\Gamma=\BE[(S^e_1)_i(S^e_1)_j]_{1\leq i,j \leq 2}$
where $(S^e_1)_1,(S^e_1)_2$ denote the first and the second coordinate
of $S^e_1$, respectively.
Therefore, note that 
\[
\BP(S^e_1=o)=\BP(S_2=o)=\frac{1}{4},\ 
\BP(S^e_1=\pm e_1)=\BP(S^e_1=\pm e_2)=\BP(S_2=(1,1))=\frac{1}{8}
\] 
and that all four remaining probabilities equal $1/16.$ 
We then see that 
\[
\BE[(S^e_1)_1(S^e_1)_1]
=\BP((S^e_1)_1=\pm 1)=1-\BP(S^e_1=o)-\BP(S^e_1=\pm e_2)
=\frac{1}{2},
\]
and by symmetry that $\BE[(S^e_1)_2(S^e_1)_2]=\frac{1}{2}.$ 
Furthermore, 
\[
\BE[(S^e_1)_1(S^e_1)_2]=1\BP(S^e_1=\pm(1,1))-1\BP(\pm (1,-1))=0,
\]
and so 
\[
\Gamma
=\frac{1}{2}\left[ \begin{array}{cc}
1 & 0  \\
0 & 1
\end{array}
\right]
\textrm{ so that }
\Gamma^{-1}
=2\left[ \begin{array}{cc}
1 & 0  \\
0 & 1
\end{array}
\right].
\]
Therefore 
\[
\sqrt{\det \Gamma} =\sqrt{\frac{1}{4}}=\frac{1}{2}
\textrm{ and } 
\CJ^*(y)^2=\Vert y \cdot \Gamma^{-1} y \Vert^2=
4  \Vert  y \Vert^2.
\]
Furthermore, using the definition of the map $T$ from 
\eqref{eqn:defSne} we see that 
\[
\Vert  T(x) \Vert^2
=\Vert  \frac{1}{2}(x_1+x_2,x_2-x_1) \Vert^2
=\frac{1}{4}((x_1+x_2)^2+(x_2-x_1)^2)=\frac{\Vert x \Vert^2}{2}.
\]
Inserting this into \eqref{eqn:Psnxeven} we obtain
\begin{eqnarray}
\BP(S_n=x)\leq\frac{1}{\pi n}e^{-\frac{4\Vert T(x) \Vert^2}{2n}}
+\frac{C'}{n\Vert T(x) \Vert^2} 
=\frac{1}{\pi n}e^{-\frac{\Vert x \Vert^2}{n}}
+\frac{2C'}{n\Vert x \Vert^2}.
\nonumber 
\end{eqnarray}
Lastly, we have that $\Vert x\Vert^2 \leq |x|^2 \leq 2\Vert x \Vert^2$
and so we conclude that 
\[
\BP(S_n=x)
\leq \frac{1}{\pi n}e^{-\frac{\vert x \vert^2}{2n}}
+\frac{4C'}{n\vert x \vert^2}.
\]
Thus, \eqref{eqn:2dLCLT} holds in this case for any $C \geq 4C'$.

Assume now instead that $|x|$ is odd so that $\BP(S_n=x)=0$ whenever $n$ is even
or $|x|>n.$
For $n \geq |x|$, we can now sum over the neighbors $y \sim x$ and use the last inequality for each $y$ to obtain, 
for $|x|\geq 3,$ 
\begin{eqnarray*}
\lefteqn{\BP(S_n=x)
=\frac{1}{4}\sum_{y\sim x} \BP(S_{n-1}=y)}\\
& & \leq \frac{1}{4}\sum_{y\sim x} \left(\frac{1}{\pi (n-1)}
e^{-\frac{\vert y \vert^2}{2(n-1)}}
+\frac{4C'}{(n-1)\vert y \vert^2}\right) 
\leq \frac{1}{\pi (n-1)}e^{-\frac{(\vert x \vert-1)^2}{2(n-1)}}
+\frac{4C'}{(n-1)(\vert x \vert-1)^2} \\
& & \leq \frac{2}{\pi n}e^{-\frac{(\vert x \vert-1)^2}{2n}}
+\frac{C}{n\vert x \vert^2}
\leq \frac{2}{\pi n}e^{-\frac{\vert x \vert^2}{2n}}e^{\frac{\vert x \vert}{n}}
+\frac{C}{n\vert x \vert^2} 
\leq \frac{2}{n}e^{-\frac{\vert x \vert^2}{2n}}
+\frac{C}{n\vert x \vert^2}
\end{eqnarray*}
for some $C<\infty$, where, in the last inequality, we have used that $n \geq |x|$.
This establishes \eqref{eqn:2dLCLT}
also for $|x|$ odd.
\end{proof}

\begin{proof}[Proof of Lemma \ref{lemma:sumWoxest}.] 
According to Lemma \ref{lemma:2dLCLT} we have that for $|x|\geq 3$, 
\[
\BP(S_n=x)\leq \frac{2}{n}e^{-\frac{|x|^2}{2n}}+\frac{C}{n|x|^2}.
\]
Observe next that the function $ye^{-cy}$ is decreasing in $y$ for 
$y>1/c$ from which it follows that
$\frac{2}{n}e^{-\frac{|x|^2}{2n}}$ is increasing in $n$ for 
$n<\frac{|x|^2}{2}.$ 
To prove \eqref{eqn:sumWoxest2}, we then observe that
\begin{eqnarray*}
\lefteqn{\sum_{n=|x|}^{\left\lfloor \frac{|x|^2}{2\log |x|}\right\rfloor}
\left(\frac{1}{4+\kappa}\right)^{n} W_n^{o,x}
=\sum_{n=|x|}^{\left\lfloor \frac{|x|^2}{2\log |x|}\right\rfloor}
\left(\frac{4}{4+\kappa}\right)^{n} 4^{-n} W_n^{o,x}
\leq \sum_{n=|x|}^{\left\lfloor \frac{|x|^2}{2\log |x|}\right\rfloor} \BP(S_n=x)}\\
& & \leq \sum_{n=|x|}^{\left\lfloor \frac{|x|^2}{2\log |x|}\right\rfloor}
\left(\frac{2}{n}e^{-\frac{|x|^2}{2n}}
+\frac{C}{n|x|^2}\right)
 \leq \sum_{n=|x|}^{\left\lfloor \frac{|x|^2}{2\log |x|}\right\rfloor}
\left(\frac{4\log |x|}{|x|^2}e^{-\frac{|x|^22\log |x|}{2|x|^2}} 
+\frac{C}{|x|^3}\right)\\
& & \leq \frac{|x|^2}{2 \log |x|}
\left(\frac{4\log |x|}{|x|^2}e^{-\log |x|} 
+\frac{C}{|x|^3}\right)
= 2|x|^{-1}+C\frac{1}{2|x|\log |x|}
\leq 3|x|^{-1},
\end{eqnarray*}
where the last inequality follows by taking $|x|$ 
sufficiently large.

In order to prove \eqref{eqn:sumWoxest3}, observe first that 
\begin{eqnarray} \label{eqn:sumWoxest31}
\lefteqn{\sum_{n=\left\lfloor \frac{|x|^2}{2\log |x|}\right\rfloor+1}^{|x|^2}
\left(\frac{1}{4+\kappa}\right)^{n} W_n^{o,x}
=\sum_{n=\left\lfloor \frac{|x|^2}{2\log |x|}\right\rfloor+1}^{|x|^2} 
\left(\frac{4}{4+\kappa}\right)^{n} 4^{-n} W_n^{o,x}}\\
& & =\sum_{n=\left\lfloor \frac{|x|^2}{2\log |x|}\right\rfloor+1}^{|x|^2} 
\left(1-\frac{\kappa}{4+\kappa}\right)^{n} \BP(S_n=x)
\leq \left(1-\frac{\kappa}{4+\kappa}\right)^{\frac{|x|^2}{2 \log |x|}}
\sum_{n=\left\lfloor \frac{|x|^2}{2\log |x|}\right\rfloor+1}^{|x|^2}\BP(S_n=x). \nonumber
\end{eqnarray}
Next, by using Lemma \ref{lemma:2dLCLT} and the fact that $ye^{-cy}$ 
is maximized when $y=1/c$ whenever $c>0,$ we see that  
\begin{eqnarray} \label{eqn:sumWoxest32}
\lefteqn{\sum_{n=\left\lfloor \frac{|x|^2}{2\log |x|}\right\rfloor+1}^{|x|^2}
\BP(S_n=x)\leq\sum_{n=\left\lfloor \frac{|x|^2}{2\log |x|}\right\rfloor+1}^{|x|^2}
\left(\frac{2}{n}e^{-\frac{|x|^2}{2n}}
+\frac{C}{n|x|^2}\right)}\\
& & \leq \sum_{n=\left\lfloor \frac{|x|^2}{2\log |x|}\right\rfloor+1}^{|x|^2}
\left(\frac{4}{|x|^2}e^{-1}
+\frac{2C\log |x|}{|x|^4}\right)
\leq |x|^2\left(\frac{4}{|x|^2}e^{-1}
+\frac{2C\log |x|}{|x|^4}\right)\leq 2, \nonumber
\end{eqnarray}
by taking $|x|$ large enough.
Combining \eqref{eqn:sumWoxest31} and \eqref{eqn:sumWoxest32}
proves \eqref{eqn:sumWoxest3}.
\end{proof}

The last proof of this Appendix is that of Lemma \ref{lemma:Go11lowest}.

\begin{proof}[Proof of Lemma \ref{lemma:Go11lowest}] 
We will be somewhat informal at the start since the proof relies 
on a well known technique for two-dimensional random walks. 
Indeed, using the so-called ''45-degree-trick'' 
it is easy to show that 
\[
W_{2n}^{o,(1,1)}={2n \choose n} {2n \choose n+1}.
\]
Informally, this can be explained as follows. Start by considering 
a clockwise rotation of $\BZ^2$ by 45 degrees. 
After re-orientation, 
the original walk can now be viewed as two independent one-dimensional
random walks of lengths $2n$ on this rotated lattice.
In order for the original random walk
to end up at $(1,1)$, the re-oriented walk must be such that the 
vertical walker returns to the origin at time $2n$, 
while the horizontal must take $n+1$ steps to the right and 
$n-1$ to the left, in order to end up at the 
position two steps to the right of the origin at time $2n.$

Continuing, we have that 
${2n \choose n+1}=\frac{n}{n+1} {2n \choose n}$
so that 
\[
W_{2n}^{o,(1,1)}=\frac{n}{n+1} {2n \choose n}^2.
\]
As in the proof of Lemma \ref{lemma:partialsumlowest} we then
see that 
\[
W_{2n}^{o,(1,1)} \geq \frac{n}{n+1} \frac{4^{2n}}{\pi n}e^{-1/(3n)}
=\frac{4^{2n}}{\pi (n+1)}e^{-1/(3n)}.
\]
We can therefore conclude that 
\begin{equation} \label{eqn:sumest_alt}
\sum_{n=1}^{N-1} \left(\frac{1}{4+\kappa}\right)^{2n}W^{o,(1,1)}_{2n}
\geq
\sum_{n=1}^{N-1} \left(\frac{4}{4+\kappa}\right)^{2n} 
\frac{e^{-1/(3n)}}{ \pi (n+1)}.
\end{equation}
Next, let $f(x)=\frac{e^{-1/(3x)}}{x+1}$ and observe that 
\[
f'(x)=\frac{1}{x+1}\frac{e^{-1/(3x)}}{3x^2}-\frac{1}{(x+1)^2}e^{-1/(3x)}
= \frac{e^{-1/(3x)}}{x+1} \left( \frac{1}{3x^2}-\frac{1}{x+1} \right),
\]
and so $f(x)$ is decreasing for all $x\geq 1.$ Therefore,
\begin{eqnarray} \label{eqn:partialsumlowintest_alt}
\lefteqn{\sum_{n=1}^{N-1} \frac{e^{-1/(3n)}}{\pi (n+1)}
\left(\frac{4}{4+\kappa}\right)^{2n}
\geq \sum_{n=1}^{N-1} \int_{n}^{n+1} \frac{e^{-1/(3x)}}{ \pi (x+1)}
\left(\frac{1}{1+\kappa/4}\right)^{2x} dx }\\
& & = \int_1^{N} \frac{e^{-1/(3x)}}{ \pi (x+1)}
\left(\frac{1}{1+\kappa/4}\right)^{2x} dx 
=\int_{\kappa}^{N \kappa} \frac{e^{-\kappa/(3y)}}{\pi (y/\kappa+1)}
\left(\frac{1}{1+\kappa/4}\right)^{2y/\kappa} 
\frac{1}{\kappa}dy \nonumber \\
& & =\int_{\kappa}^{N \kappa} \frac{e^{-\kappa/(3y)}}{\pi (y+\kappa)}
\left(\frac{1}{1+\kappa/4}\right)^{2y/\kappa} dy. 
\nonumber
\end{eqnarray}
Then, observe that for any $x>0$ we have that $\log (1+x)\leq 2x$.
Thus, since $e^{-x}\geq 1-x$ for every $x>0$,
\[
\left(\frac{1}{1+\kappa/4}\right)^{2y/\kappa}
=\exp\left(-\frac{2y}{\kappa}\log(1+\kappa/4)\right)
\geq e^{-y} \geq 1-y
\]
and so,
\begin{eqnarray}\label{eqn:intest_alt}
\lefteqn{\int_{\kappa}^{N\kappa} \frac{e^{-\kappa/(3y)}}{\pi (y+\kappa)}
\left(\frac{1}{1+\kappa/4}\right)^{2y/\kappa} dy 
\geq \int_{\kappa}^{N \kappa} \frac{1}{\pi (y+\kappa)}
\left(1-y\right)\left(1-\frac{\kappa}{3y}\right)dy }\\
& & =\frac{1}{\pi}\int_{\kappa}^{N \kappa} \frac{3+\kappa}{3(y+\kappa)}
-\frac{y}{y+\kappa}-\frac{\kappa}{3 y (y+\kappa)}dy 
\geq \frac{1}{\pi}\int_{\kappa}^{N \kappa} \frac{1}{y+\kappa}
-1-\frac{\kappa}{3 y^2}dy 
\nonumber \\
& &  =\frac{1}{\pi}\left[\log(y+\kappa)\right]_\kappa^{N\kappa}
-\frac{(N-1)\kappa}{\pi}
+\left[\frac{\kappa}{3 \pi y}\right]_\kappa^{N\kappa}
=  \frac{\log(N+1)}{\pi}
-\frac{(N-1)\kappa}{\pi}
+\frac{1}{3\pi N}
-\frac{1}{3\pi}\nonumber \\
& & \geq \frac{\log N}{\pi}
-\frac{(N-1)\kappa}{\pi}-\frac{1}{3\pi}.
\nonumber
\end{eqnarray}
Combining \eqref{eqn:sumest_alt}, \eqref{eqn:partialsumlowintest_alt} 
and \eqref{eqn:intest_alt} gives us that 
\[
\sum_{n=1}^{N-1} \left(\frac{1}{4+\kappa}\right)^{2n}W^{o,(1,1)}_{2n}
\geq \frac{\log N}{\pi}
-\frac{(N-1)\kappa}{\pi}-\frac{1}{3 \pi }.
\]
We therefore conclude that 
\begin{eqnarray*} 
\lefteqn{G^{o,(1,1)}
=\sum_{n=1}^{\infty} \left(\frac{1}{4+\kappa}\right)^{2n}
W^{o,(1,1)}_{2n} 
\geq \sum_{n=0}^{\lfloor \kappa^{-1}\rfloor} 
\left(\frac{1}{4+\kappa}\right)^{2n}W^{o,(1,1)}_{2n} }\\
& & \geq 
\frac{\log (\lfloor\kappa^{-1}\rfloor+1)}{\pi}
-\frac{\lfloor \kappa^{-1}\rfloor \kappa}{\pi}-\frac{1}{3 \pi }
\geq \frac{\log \kappa^{-1}}{\pi}-\frac{1}{\pi}-\frac{1}{3 \pi}
\geq \frac{\log \kappa^{-1}}{\pi}-1
\end{eqnarray*}
as desired.
\end{proof}

{\bf Acknowledgments:} We would like to thank the anonymous referees for 
providing many helpful and detailed comments. We would also like to thank Yves Le Jan 
for suggesting the problem studied in this paper and Stas Volkov for discussions
concerning random walks.

\end{appendices}

\end{document}